\documentclass{amsart}
\usepackage{graphics}
\usepackage{graphicx}
\let\lineref\relax
\usepackage{myalgo}
\usepackage{algorithm}
\usepackage{algorithmicx}
\usepackage{amsfonts,amssymb,amsopn,amscd,amsmath}
 \usepackage{mathrsfs}

\usepackage{xcolor}
\usepackage{pagecolor}
\usepackage{lipsum}  
\definecolor{bl}{HTML}{002b36}
\definecolor{wh}{HTML}{657b83}

\usepackage[pagebackref,colorlinks,linkcolor=bordeaux,citecolor=darkblue,urlcolor=black,hypertexnames=true]{hyperref}
 \newtheorem{thm}{Theorem}[section]
 
 \newtheorem{lem}[thm]{Lemma}{\rm}
 \newtheorem{prop}[thm]{Proposition}

 \newtheorem{defn}[thm]{Definition}{\rm}
 \newtheorem{idpc}[thm]{IDPC}{\rm}
 \newtheorem{prog}[thm]{Program constant}{\rm}
 \newtheorem{elim}[thm]{Elimination rule}{\rm}
 
 \newtheorem{rem}[thm]{Remark}
 
\numberwithin{equation}{section}

\DeclareMathOperator{\grad}{grad}

\begin{document}
	\definecolor{darkblue}{rgb}{0.0, 0.0, 0.55}
	\definecolor{bordeaux}{rgb}{0.34, 0.01, 0.1}
\newcommand{\spc}{\ensuremath{\,\,}}
\def\red{\color{red}}
\def\bl{\color{blue}}
\def\ora{\color{orange}}
\def\green{\color{green}}
\def\br{\color{brown}}
\newcommand{\clist}{\texttt{c\_list}}
\newcommand{\calpha}{\texttt{c\_alpha}}
\newcommand{\slist}{\texttt{s\_list}}
\def\la{\langle}
\def\ra{\rangle}
\def\e{{\rm e}}
\def\x{\mathbf{x}}
\def\by{\mathbf{y}}
\def\bz{\mathbf{z}}
\def\F{\mathcal{F}}
\def\R{\mathbb{R}}
\def\T{\mathbf{T}}
\def\N{\mathbb{N}}
\def\K{\mathbf{K}}
\def\bK{\overline{\mathbf{K}}}
\def\Q{\mathbf{Q}}
\def\M{\mathbf{M}}
\def\O{\mathbf{O}}
\def\C{\mathbf{C}}
\def\P{\mathbf{P}}
\def\Z{\mathbb{Z}}
\def\H{\mathcal{H}}
\def\A{\mathbf{A}}
\def\V{\mathbf{V}}
\def\AA{\overline{\mathbf{A}}}
\def\B{\mathbf{B}}
\def\L{\mathbf{L}}
\def\bS{\mathbf{S}}
\def\H{\mathcal{H}}
\def\I{\mathbf{I}}
\def\Y{\mathbf{Y}}
\def\X{\mathbf{X}}
\def\G{\mathbf{G}}
\def\L{\mathbf{L}}
\def\D{\mathbf{D}}
\def\B{\mathbf{B}}
\def\f{\mathbf{f}}
\def\z{\mathbf{z}}
\def\y{\mathbf{y}}
\def\d{\hat{d}}
\def\bx{\mathbf{x}}
\def\y{\mathbf{y}}
\def\v{\mathbf{v}}
\def\g{\mathbf{g}}
\def\w{\mathbf{w}}
\def\b{\mathcal{B}}
\def\a{\mathbf{a}}
\def\q{\mathbf{q}}
\def\u{\mathbf{u}}
\def\s{\mathcal{S}}
\def\cc{\mathcal{C}}
\def\co{{\rm co}\,}
\def\cp{{\rm CP}}
\def\tg{\tilde{f}}
\def\tx{\tilde{\x}}
\def\supmu{{\rm supp}\,\mu}
\def\supnu{{\rm supp}\,\nu}
\def\m{\mathcal{M}}
\def\bR{\mathbf{R}}
\def\om{\mathbf{\Omega}}
\def\s{\mathcal{S}}
\def\k{\mathcal{K}}
\def\la{\langle}
\def\ra{\rangle}
\def\blambda{{\boldmath{\lambda}}}
\def\bsmlambda{\boldmath{\lambda}}
\def\ov{\overline{o}}
\def\und{\underline{o}}
\def\minlog{\texttt{Minlog}}
\newcommand{\newjar}[1]{{{\color{blue}#1}}}
\newcommand{\comment}[3]{%
\ifcomment%
	{\color{#1}\bfseries\sffamily(#3)%
	}%
	\marginpar{\textcolor{#1}{\hspace{3em}\bfseries\sffamily #2}}%
	\else%
	\fi%
}

\newcommand{\comm}[3]{%
	\ifcomm%
	{\color{#1} #3}
	\else%
	\fi%
}

\newcommand{\victor}[1]{
	\comment{blue}{V}{#1}
}

\newcommand{\answer}[1]{
	\comm{black}{A}{#1}
}
\newif\ifcomm
\commfalse
\commtrue

\newif\ifcomment
\commentfalse
\commenttrue
\title[Computer-assisted proofs for Lyapunov stability via SOS certificates]{Computer-assisted proofs for Lyapunov stability via Sums of Squares certificates and Constructive Analysis}

\author[G. Devadze \and V. Magron  \and S. Streif]{Grigory Devadze \and Victor Magron \and Stefan Streif}

\address{Grigory Devadze: Technische Universit\"at Chemnitz\\
Automatic Control and System Dynamics Lab\\
09107 Chemnitz, Germany}
\email{grigory.devadze@etit.tu-chemnitz.de}
\address{Victor Magron: LAAS, 7 avenue du Colonel Roche\\
31077 Toulouse C\'edex 4,France}
\email{vmagron@laas.fr}
\address{Stefan Streif: Technische Universit\"at Chemnitz\\
Automatic Control and System Dynamics Lab\\
09107 Chemnitz, Germany}
\email{stefan.streif@etit.tu-chemnitz.de}
\date{}

\begin{abstract}
We provide a computer-assisted approach to ensure that a given discrete-time polynomial system is (asymptotically) stable.
Our framework relies on constructive analysis together with formally certified sums of squares Lyapunov functions. The crucial steps are formalized within the proof assistant $\minlog$. 
We illustrate our approach with various examples issued from the control system literature. 
\end{abstract}



\maketitle

\section{Introduction}
Computer-assisted theorem proving -- performed by special software called {\em proof assistants} -- have been successfully used for important problems, including the formalization of  Kepler's conjecture  \cite{pi17} and a proof of the odd order theorem in group theory \cite{gonthier2013machine}. One of the main advantages of using the proof assistants is their ability of the computer-aided checking of the formal correctness of the logical statements.

Lyapunov stability plays one of the most crucial roles in nonlinear control theory and the {\em second method of Lyapunov}, also called the {\em Lyapunov function method} is a standard tool for the design of controllers.
In this work we present a computer-assisted verification framework for the Lyapunov stability analysis of polynomial dynamical. 

We consider the problem of certifying the stability of the autonomous discrete-time dynamical system 
 \begin{align} \label{def:dyn-sys}
 	\pmb x(n+1) = f(\pmb x(n))
 \end{align}
 where $f:D \rightarrow \mathbb{R}^N$ is a continuous function with equilibrium point at origin and $D \subset \mathbb{R}^N$.
 
 In the presented approach, first, Lyapunov functions are constructed via sums of squares techniques; second, the formal correctness is certified within of the framework of \texttt{Minlog}. The certificate guarantees that the constructed Lyapunov function meets its specification and the combination of both approaches allows for automation of reasoning on dynamical systems. We provide the code as open source \cite{con2formlog} libraries for \texttt{Minlog}. The stability analysis requires dealing with nonlinear inequalities and in contrast to purely numerical methods \cite{Giesl2015-review}, we are interested in reasoning with the so-called exact real arithmetic. For this purpose the standard library of \texttt{Minlog} is extended in this work by several necessary definitions and theorems in the setting of constructive reals.
 The presented concepts may also be viewed as a preliminary work for the development of formalized nonlinear control theory, hence offering the possibility of computer-assisted controller design and program extraction for controller implementation.
\subsection*{Related works in the formalization of dynamical systems}
%
%

In the context of dynamical systems,  \cite{boldo2013wave} formally proves the correctness of a program implementing the numerical resolution of a wave equation, with Frama-C \cite{cuoq2012frama}.  
In \cite{immler2019flow}, the authors formalize  variational equations related to the solutions of ordinary differential equations in Isabelle/HOL \cite{Nipkow02IsabelleHol}. 
Beyond formalization of numerical analysis procedures, several frameworks are currently available to perform computer-assisted proofs of mathematical analysis. 
The work in \cite{affeldt2020formalizing} proposes formalized mathematical structures in Coq \cite{CoqProofAssistant} for function  spaces. 
Earlier developments have been pursued towards the formalization of real analysis in several proof assistants, including Coq \cite{mayero2001formalisation}, HOL-Light,  \cite{harrison2012theorem}, Isabelle/HOL \cite{fleuriot2000mechanization} or  ACL2 \cite{gamboa2001nonstandard}, as well as constructive real analysis \cite{geuvers2000constructive,Krebbers2013-realanalysis}.
Proof assistants also play an important role in formal verification of control systems. 
Zou \cite{Zou2013-form-ver-Chinese-train} used Isabelle for implementation of their control system formal verification based on Hybrid Hoare logic. 
The Why3 platform \cite{bobot2011why3}  coupled with MATLAB/Simulink was used in \cite{Araiza-Illan2014-thm-prover-sys, Araiza2015-thm-proving-Simulink} to perform simple stability checks of linear discrete systems with quadratic Lyapunov functions. 
Gao \cite{Gao2014-descr-ctrl-thr} mentioned several proof assistant software tools to be considered for implementation of formal verification of control systems, including Coq \cite{CoqProofAssistant}, HOL-Light \cite{harrison1996hol}, Isabelle/HOL and Lean \cite{de2015lean}. The most closely related work is due to \cite{Anand2015-roscoq} where the authors present the framework for the certified programs for robots based on constructive real numbers \cite{Oconnor2008-certified}. Their work is based on the Logic of Events for distributed agent systems and generation of verified software controllers. However they do not address the fundamental principles of control theory and system dynamics like stability, robustness or optimality. In \cite{Rouhling2018-formal} the author formally proves the soundness of a control function for the inverted pendulum within the classical logic based on the prior formalization of Lassalle's principle \cite{cohen2017formal}. The author in \cite{Rouhling2018-formal} also mentions the impossibility of generating and running programs with that framework.
\subsection*{Related works in the formalization of constructive analysis}

To the authors best knowledge there exist two important formalizations of Bishop-like constructive analysis: \texttt{C-CORN} in proof assistant Coq \cite{cruz2004c} and the standard library of constructive reals in proof assistant $\minlog$ \cite{schwichtenberg2011minlog}.
The Coq proof assistant is based on the calculus of inductive constructions which is a dependent type theory with inductive and co-inductive types. The important part of \texttt{C-CORN} library is the type of constructive reals \texttt{CR} which has been created with computationally feasible execution in mind \cite{Oconnor2008-certified}. The constructive reals are defined over a general concept of a so-called regular function, i.e., a function $x: \mathbb{Q}^+ \rightarrow \mathbb{Q} $ is called regular if
\[
\forall \varepsilon_1,\varepsilon_2 \spc |x(\varepsilon_1) - x(\varepsilon_2)| \le \varepsilon_1 + \varepsilon_2.
\]
Using the regular functions principle, a completion of any metric space $X$ (denoted as $\mathcal{C}(X)$) can be constructed \cite{OConnor2007-monad}. It is shown that the completion  forms a monad and the set of constructive real numbers is defined as $\mathbb{R}:=\mathcal{C}(\mathbb{Q})$. Thus, all usual operations on real numbers are defined as uniformly continuous functions from $\mathbb{Q}$ to $\mathbb{Q}$ and completed afterwards using the completion monad to functions on $\mathbb{R}$. Such application of completion monad is often called lifting. The \texttt{C-CORN} library is a large code base consisting of formalization of elementary transcendental functions, power series, differentiation, integration and implementation of Picard-Lindel\"of theorem for ODE \cite{Makarov2013-picard}. Still the program extraction mechanism is not a straightforward task due to the complexity in terms of the depth of definitions, which is required when dealing with constructive real numbers \cite{Krebbers2013-realanalysis}.

In the present work the interactive proof system $\minlog$ \cite{schwichtenberg2011minlog,Berger2011-minlog} is used for verification. 
The development of $\minlog$ adresses proof-theoretical methods based on realizability and their applications to the verification and especially to the program extraction. 
The implementation of $\minlog$ is done in \texttt{Scheme}. The main reason which justifies the usage of $\minlog$ is its special treatment of the program extraction, i.e., in addition to the program extraction, $\minlog$ provides the proof of correctness of the extracted program. 
The machinery behind $\minlog$ is supported by the soundness theorem and the theory of computable functionals \cite{Berger2011-minlog,Berger2012-proofs}. 
The conversion of proofs to instances handled by the soundness theorem can be done automatically \cite{Berger2011-minlog,Miyamoto2013-program,Benl1998-minlog}. This kind of approach is not used in, e.g., Coq \cite{CoqProofAssistant}, where the program extraction is guided by the use of external tools. Due to the so-called normalization-and-evaluation technique, $\minlog$ achieves often more simpler and efficient terms, which matters in practice \cite{Berger2011-minlog}.
As for constructive analysis the standard $\minlog $ library consists of formalization of real numbers as Cauchy sequences of rationals, uniformly continuous functions and Intermediate Value Theorem. As all uniformly continuous functions are represented as maps from $\mathbb{Q}$ to $\mathbb{Q}$, a similar lifting technique is used, but without the formalization of the notion of metric spaces. Even though the $\minlog$ library misses important basic notions, we found it promising for our purpose, as it presents an easy framework based on natural deduction. We refer to Section \ref{sec:minlog} for more details.

\subsection*{Related works in certification of polynomial systems}
A related result~\cite{khalil2002nonlinear} states that the equilibrium of the system (\ref{def:dyn-sys}) is asymptotically stable if there is some positive definite function $V(\x)$ such that the difference of successive values of $V$ along the trajectories of the system is negative definite, or stated equivalently, that the opposite of this difference is positive definite.
If we replace the later condition by the weaker condition that the derivative is nonnegative, then the system is stable.
In both cases, such a function $V$ is called a {\em Lyapunov function}.
For polynomial systems, i.e., when $f$ is a polynomial, these two sufficient positivity/nonnegativity conditions can be strengthened by computing a positive definite function $V$ which is a {\em sum of squares} (SOS) of polynomials and such that $- \Delta V$ is SOS~\cite{phdParrilo}.
In this case, we refer to $V$ as an {\em SOS Lyapunov function}.\\
Proving stability with SOS polynomials lies in the research track of characterizing nonlinear systems through linear programs (LP), whose unknown variables are measures with support on the system constraints. 
This methodology was first introduced for static polynomial optimization problems~\cite{lasserre2001global,phdParrilo}. 
In~\cite{lasserre2001global}, Lasserre provides a  hierarchy of relaxations yielding a converging sequence of lower bounds for the minimum of a polynomial over a set of constraints also defined with polynomials. 
Each bound is obtained by solving a semidefinite program (SDP) \cite{vandenberghe1996semidefinite}, i.e., by computing the infimum of a linear function under the constraint that a symmetric matrix has only nonnegative eigenvalues. 
SDP has by its own many applications, for instance in combinatorial optimization~\cite{Gvozdenovic09}, control theory~\cite{BEFB94} or parameter estimation \cite{Streif2013-estimation}. 
Available software libraries such as SeDuMi~\cite{Sturm98usingsedumi}, SDPA~\cite{Yamashita10SDPA} or Mosek~\cite{moseksoft} can solve SDP with up to a few thousand variables.
The general idea consists of modeling a polynomial optimization problem as an infinite-dimensional LP over Borel measures. 
In practice, one approximates the solution of this LP with a converging SDP primal-dual programs, called {\em moment-SOS} or Lasserre's  hierarchy.
The primal is a \emph{moment} problem, that is an optimization problem where variables are the moments of a Borel measure. The first moment is related to the mean values, the second moment is related to the variances, etc.  
The book~\cite{lasserre2009moments} provides a non-exhaustive list of various problems arising in applied mathematics, statistics and probability, economics, engineering which can be cast as instances of optimization problems over  moments. 
The dual is an SOS problem, where the variables are the coefficients of SOS polynomials. An example of SOS polynomial is $(\frac{1}{3} x_1 - x_2)^2 + (x_1+x_2)^2$. 
Even though there exist positive polynomials which cannot be written with SOS decompositions, it turns out that when the set of constraints satisfies certain assumptions (slightly stronger than compactness) then one can represent positive polynomials with weighted SOS decompositions. 
These weights happen to be the polynomials defining the set of constraints. Among several similar powerful results for representation of positive polynomials, the one related to the dual of Lasserre's hierarchy is due to Putinar~\cite{Putinar1993positive}.
For dynamical systems, this hierarchy have been extended multiple times, for example to obtain converging approximations for optimal control problems~\cite{lasserre2008nonlinear}, reachable sets~\cite{magron2019semidefinite} or invariant measures~\cite{invsdp}.\\
However, such methods rely solely on numerical SDP solvers, implemented in double precision, thus output approximate nonnegativity certificates. 
In order to circumvent the numerical errors, one can use several algorithmic frameworks either in the univariate case~\cite{Chevillard11,univsos} or in the multivariate case~\cite{PaPe08,KLYZ08,magron2018exact}.
In particular, 
the rounding-projection algorithm from \cite{PaPe08} outputs a weighted rational SOS decompositon of polynomials lying in the interior of the SOS cone.
More recently, \cite{magron2018exact} has provided an alternative framework to
compute such decompositions. 
In a nutshell, this framework first outputs an
approximate SOS decomposition for a perturbation of the input polynomial with an
arbitrary-precision SDP solver. 
Then an exact decomposition is provided thanks to
the perturbation terms and a compensation phenomenon of the numerical errors.
The underlying algorithm is implemented in the RealCertify Maple package \cite{magron2018realcertify}.\\
Beyond the stability proofs of critical control systems, computing exact certificates of  nonnegativity is mandatory in many application fields, such as certified evaluation in computer arithmetics~\cite{Chevillard11,martin2017reflexive} or 
formal verification of real inequalities~\cite{magron14,harrison2007verifying,Monniaux-Positivstellensatz2011} within proof 
assistants such as Coq~\cite{CoqProofAssistant} or HOL-Light~\cite{harrison1996hol}.

\subsection*{Contributions}
The aim of the present work is to provide  computer-assisted proofs for the stability of a given discrete-time polynomial system.  
For such a system, our framework consists of proving formally that a given function is a Lyapunov function.
Our formal proofs of polynomial nonnegativity are handled with weighted SOS certificates. 
Since proof assistants such as \texttt{MinLog} have computational limitation, we rely on external tools that produce certificates, whose checking is reasonably easier from a computational point of view. 
These certificates are obtained with the RealCertify \cite{magron2018realcertify} tool available outside of the proof assistant \texttt{MinLog} and  verified inside. 
Our first contribution, presented in Section~\ref{sec:prelim_exact}, is to provide an algorithm, called \texttt{exactLyapunov} to compute exact weighted rational SOS Lyapunov functions for a given discrete-time polynomial system (\ref{def:dyn-sys}).
For the sake of clarity, we provide detailed explanation for the case of discrete systems only. Analogously, 	the results can be applied to the continuous-time systems if we require that the Lyapunov function decreases along the system's trajectories. 
Section \ref{sec:minlog} is dedicated to the presentation of the $\minlog$ formalization framework and its extension of the standard library. 
Our further contribution, provided in Section \ref{sec:formal}, is to derive a framework to formally prove the stability of such  systems, by verifying the output of  \texttt{exactLyapunov} within the $\minlog$ system.
We illustrate our formal framework in Section \ref{sec:benchs}. 

\section{Preliminary Background and Results}
\label{sec:prelim}
\subsection{Positive polynomials, sums of squares}
\label{sec:prelim_sos}
Let $\mathbb{Q}$ (resp.~$\mathbb{R}$) be the field of rational (resp.~real) numbers.
With $d,N \in \N$, let us denote by $\R[\x]$ (resp. $\R[\x]_{d}$) the space of real-valued $N$-variate polynomials (resp. of degree at most $d$) in the variable $\x=(x_1,\ldots,x_N) \in \R^N$.
Given a real symmetric matrix $\G$, the notation $\G \succeq 0$ (resp. $\G \succ 0$) means that $\G$ is \emph{positive semidefinite} (resp. positive definite), that is $\G$ has only nonnegative (resp. positive) eigenvalues.
Let us note $\Sigma[\x]$ the set of sums of squares (SOS) of polynomials and $\Sigma[\x]_{2 d} := \Sigma[\x] \cap \R[\x]_{2 d}$ the set of SOS polynomials of degree at most $2 d$.
An $N$-variate polynomial $g$ belongs to $\Sigma[\x]_{2 d}$ if there exist a nonzero $r \in \N$ and $g_1,\dots,g_r \in \R[\x]_{d}$ such that $g = g_1^2 + \dots + g_r^2$.
An example of SOS polynomial is $g = 4 x_1^4 + 4
x_1^3 x_2 - 7 x_1^2 x_2^2 - 2 x_1 x_2^3 + 10 x_2^4 = (2 x_1 x_2 + x_2^2)^2 + (2
x_1^2 + x_1 x_2 - 3 x_2^2)^2$.
For each multi-index $\alpha = (\alpha_1,\dots,\alpha_N) \in \N^N$, we use the notation $\x^{\alpha} = x_1^{\alpha_1} \cdots x_1^{\alpha_N}$, and we write $g(\x) = \sum_{\alpha} g_{\alpha} \x^{\alpha}$.
The support of a polynomial $g$ written as above is the set $\{\alpha \in \N^N: g_\alpha \neq 0 \}$. 
For our previous example, the support is $\{(4,0), (3,1), (2,2) ,(1,3), (0,4) \}$.
We say that $g \in \R[\x]_{d}$ has {\em full support} if its support is $\N^N_d := \{ \alpha \in \N^N : \alpha_1 + \dots + \alpha_N \leq d \}$.

It is trivial that the condition of $g$ to be SOS ensures that $g$ is nonnegative over $\R^N$. 
The converse statement establishing the equivalence between nonnegative polynomials and SOS polynomials, is true only in the cases of univariate polynomials of even degree, quadratic polynomials and quartic polynomials in two variables. 
This equivalence was proven by Hilbert and we refer the interested reader to \cite{reznick2000some} and the references therein for more details.
The condition that $g \in \Sigma[\x]_{2 d}$ is equivalent to the existence of a symmetric matrix $\G \succeq 0$ such that $g(\x) = \v_d(\x)^T \G  \v_d(\x)$, where $\v_d(\x) = (1,x_1,\dots,x_N,x_1^2,x_1 x_2,\dots, x_N^2, x_1^d,\dots,x_N^d)$ is the vector of all monomials of degree at most $d$. 
The length of this vector is equal to $\binom{N+d}{N}$.
The matrix $\G$ is called a {\em Gram matrix} \cite{choi1995sums} associated to $g$ and  
computing $\G$ boils down to solving a {\em semidefinite program} (SDP for short) \cite{vandenberghe1996semidefinite}.
We refer to this technique as the {\em Gram matrix method}.
The SOS decomposition $g = g_1^2 + \dots + g_r^2$ is a by-product of a decomposition of the form $g = \v_d(\x)^T \L^T \D \L \v_d(\x)$, with $\L$ being a lower
triangular matrix with nonnegative real entries on the diagonal and
$\D$ being a diagonal matrix with nonnegative real entries, such that $\L^T \D \L = \G $.

\subsection{SOS Lyapunov functions}
\label{sec:prelim_lyapunov}
A positive definite polynomial $f$ is a polynomial which takes only positive values for nonzero input and such that $f(x) = 0$ if and only if $x = 0$.
Given a vector of polynomials $f = (f_1,\dots,f_N)$, let us now consider a discrete-time polynomial system (\ref{def:dyn-sys}).
Let $d$ be the maximal degree of the polynomials $f_1,\dots,f_N$.
The Gram matrix method can be used to certify that such a system with equilibrium at the origin is stable by computing a {\em Lyapunov} polynomial function $V$. 
More precisely, one would like to obtain $V \in \R[\x]_{2 k}$ such that (1) $V$ is  positive definite and (2)  $V - V \circ f  \in \R[\x]_{2 k d}$ is nonnegative.
Note that $V$ is required to be of even degree, otherwise it could not be a positive definite polynomial.

Regarding (1), we look for a decomposition $V(\x) = \v_{k}(\x)^T \G_1 \v_{k}(\x)$, where $\v_k$ is the vector of all monomials of degree up to $k$, and similarly for (2).
To ensure that these two conditions (1) and (2) are fulfilled, we consider the following semidefinite feasibility problem of computing $V$, $\G_1$ and $\G_2$ such that:
\begin{equation}
\label{eq:lyapunovSOS}
\tag{SDP$_k$}
\begin{cases}
V(\x) \qquad  \qquad   \ \ = \v_{k}(\x)^T \G_1 \v_{k}(\x) \,, \quad V(0) = 0 \,, \\
V(\x) - V( f(\x)) = \v_{\ell}(\x)^T \G_2 \v_{\ell}(\x) \,,\\
V \in \R[\x]_{2 k} \,, \quad \G_1, \G_2 \succ 0  \,.
\end{cases}
\end{equation}
The optimization variables of the above  feasibility program~\eqref{eq:lyapunovSOS} are the coefficients of $V$ and the entries of the matrices $\G_1$ and $\G_2$.
Any feasible solution $(V,\G_1,\G_2)$ of this SDP ensures that $V$ is positive definite and $V - V \circ f$ is nonnegative, yielding a stability proof for the initial system (\ref{def:dyn-sys}).

To solve~\eqref{eq:lyapunovSOS}, one relies on numerical solvers implemented in double precision,
for instance SeDuMi~\cite{Sturm98usingsedumi} and SDPA~\cite{Yamashita10SDPA}, or arbitrary-precision solvers such as~SDPA-GMP~\cite{Nakata10GMP}.
However, such numerical solvers provide approximate
nonnegativity certificates, so that the polynomial $V$, and the matrices $\G_1, \G_2$ do not exactly satisfy the equality constraints of \eqref{eq:lyapunovSOS}. 
%
\subsection{From approximate to exact SOS certificates}
\label{sec:prelim_exact}
To overcome these feasibility issues, one can rely on the certification framework from  \cite{magron2018exact}, implemented in the RealCertify software library \cite{magron2018realcertify}.
For certification purpose, we are mainly interested in polynomials with rational coefficients, which belong to the interior of the SOS cone $\Sigma[\x]$, denoted by $\mathring{\Sigma}[\x]$.
Given a polynomial $g \in \mathring{\Sigma}[\x]$ of degree $2d$ with full support, there always exists by \cite[Proposition 5.5]{choi1995sums} a positive definite Gram matrix $\G$ associated to $g$, i.e., such that $g = \v_d(\x)^T \G \v_d(\x)$. 
In this case, one can apply the so-called \texttt{intsos} algorithm, stated in \cite[Algorithm 1]{magron2018exact}, to compute exact weighted rational SOS decompositions for polynomials with rational coefficients, lying in $\mathring{\Sigma}[\x]$.
Namely, we obtain as output a decomposition $g = c_1 g_1^2 + \dots c_r g_r^2$, with $c_1,\dots,c_r$ being nonnegative rational numbers and $g_1,\dots,g_r$ being polynomials with rational coefficients.

For the sake of self-containedness, let us briefly recall how \texttt{intsos} works. 
The algorithm takes as input a polynomial $g$ of degree $2 d$ with rational coefficients and full support, and accuracy parameters $\delta$ and $\delta_{c}$. 
Then, the procedure consists of a perturbation step and an absorption step.
The perturbation step consists of finding a small enough positive rational number $\varepsilon$,  such that the perturbed polynomial $g_{\varepsilon}(\x) := g(\x) - \varepsilon \sum_{\alpha \in \N^N_d} \x^{2 \alpha}$ also lies in $\mathring{\Sigma}[\x]_{2 d}$.
Then, \texttt{intsos} calls an external SDP solver with accuracy $\delta$ to find an approximate Gram matrix $\tilde{\G} \succ 0$ associated to $g_{\varepsilon}$, such that $g_{\varepsilon} (\x) \simeq \v_d(\x)^T \tilde{\G} \v_d(\x)$. 
Here the size (and rank) of $\tilde{\G}$ is $r = \binom{N+d}{N}$.
Then, we obtain an approximate Cholesky's decomposition $\L \L^T$ of $\tilde{\G}$ with accuracy $\delta_c$, and for all $i=1,\dots,r$, we compute the inner product of the $i$-th row of $\L$ by 
$\v_d(\x)$ to obtain a polynomial $s_i$.
For large enough accuracy $\delta$ of the SDP solver and accuracy $\delta_c$ of the  Cholesky's decomposition, the coefficients of the remainder $u =g_{\varepsilon} - \sum_{i=1}^r s_i^2$ become small enough to be ``absorbed'': that is $u + \varepsilon \sum_{\alpha \in \N^N_d} \x^{2 \alpha}$ admits itself an SOS decomposition, and so does $g = u + \varepsilon \sum_{\alpha \in \N^N_d} \x^{2 \alpha} + \sum_{i=1}^r s_i^2$.
If \texttt{intsos} fails to perform this absorption step, then it repeats the same procedure after increasing the accuracy of the SDP solver and Cholesky's decomposition.

We now present our algorithm, called \texttt{exactLyapunov}, similar to \texttt{intsos}, to compute exact weighted rational SOS Lyapunov functions for a given discrete-time polynomial system (\ref{def:dyn-sys}).
\begin{algorithm}
\caption{\texttt{exactLyapunov}}
\label{alg:exactLyapunov}
\begin{algorithmic}[1]
\Require polynomials $f_1,\dots, f_N \in  \mathbb{Q}[\x]$ of degree less than $d$
\Ensure lists $\clist_1$, $\clist_2$ of numbers in $\mathbb{Q}$ and lists $\slist_1$, $\slist_2$ of polynomials in $\mathbb{Q}[\x]$
\State $f \gets (f_1,\dots, f_N)$
\State $k=2$ \Comment{the minimal degree of the SOS Lyapunov function}
\State ok := false
\While {not ok} \label{line:ki}
\If {\eqref{eq:lyapunovSOS} is strictly feasible, i.e., returns a solution $(V, \G_1, \G_2)$ with $\G_1, \G_2 \succ 0$} ok := true
\Else $\ k \gets k + 2$ {\Comment{in order to ensure that $V$ has even degree}}
\EndIf
\EndWhile \label{line:kf}
\State $\clist_1$, $\slist_1$ = \texttt{intsos}$(V)$ \label{line:V}
\State $\clist_2$, $\slist_2$ = \texttt{intsos}$(V - V \circ f)$ \label{line:gradV}
\State \Return $\clist_1$, $\clist_2$, $\slist_1$, $\slist_2$
\end{algorithmic}
\end{algorithm}
\begin{prop}
\label{prop:exactLyapunov}
Let $f = (f_1,\dots, f_N)$ be a vector of rational polynomials of degree at most $d$, with $d$ being odd. 
Assume that \eqref{eq:lyapunovSOS} admits a strictly feasible solution for some $k \in \N$.
Then \texttt{exactLyapunov}$(f_1,\dots,f_N)$ terminates and returns an exact rational weighted SOS Lyapunov function for the discrete-time system (\ref{def:dyn-sys}). 
\end{prop}
\begin{proof}
The loop going from line~\lineref{line:ki} to line~\lineref{line:kf} increases the search degree for the SOS Lyapunov function. 
By assumption, \eqref{eq:lyapunovSOS} admits a strictly feasible solution for some  large enough $k \in \N$, so this loop  eventually terminates. 
In particular, we obtain an SOS Lyapunov function $V$ of degree $2 k$ satisfying the constraints of  \eqref{eq:lyapunovSOS}.
Then, one applies the perturbation-absorption procedure \texttt{intsos} to obtain an exact rational weighted SOS decomposition of $V$ at  line~\lineref{line:V}, and of $(V - V \circ f)$ at line~\lineref{line:gradV}.
By \cite[Proposition 3.5]{magron2018exact},  the procedure \texttt{intsos} always terminates in both cases, since $V$ and $V - V \circ f$ have associated positive definite Gram matrices. 
The outputs $\clist_1$ and $\slist_1$ respectively contain a list of nonnegative rational numbers $[c_1,\dots,c_r]$ and rational polynomials $[s_1,\dots,s_r]$ such that $V = c_1 s_1^2 + \dots + c_r s_r^2$. 
Similarly, we obtain an exact rational weighted SOS decomposition for $(V - V \circ f)$.
\end{proof}

\section{Minlog system}
\label{sec:minlog}
In this section, we investigate how to formally prove stability of polynomial systems, by verifying the output of \texttt{exactLyapunov} within $\minlog$. We first introduce some preliminary concepts of the proof system \texttt{Minlog}, the notion of totality and the concept of inductive definitions, before we give the details how the formalization works in constructive setting. In this work the analysis and formalization are done within of the framework of constructive analysis with special focus on the concrete and efficient computational realization as well as formal verification. The constructive analysis by itself is a mathematical analysis, which deals with computational content of mathematical proofs in an informal manner. In order to obtain the certified programs there is no way around of the usage of higher-order programming languages \cite{Buss1998-handbook,geuvers2000constructive,fleuriot2000mechanization,harrison2012theorem,Krebbers2013-realanalysis}.  
\subsection{Necessary Preliminaries}
One of the possible interpretations of constructive analysis is due to Brouwer, Heyting and Kolmogorov (sometimes called BHK interpretation). 
Constructive analysis states the logical rules in a way which can be algorithmically interpreted as a mathematical proof \cite{Buss1998-handbook}. Let $P$ and $Q$ be some arbitrary formulas. The most important BHK rules are:
\begin{itemize}
	\item Proving the implication $P \Rightarrow Q$ means finding an algorithm that transforms a proof of $P$ into a proof of $Q$.
	\item To prove $\exists \gamma \spc P[\gamma]$, an algorithm computing $\gamma$ and the proof that $P[\gamma]$ holds are required.
	\item To prove $\forall \gamma \in A \spc P[\gamma]$, an algorithm is required that converts $\gamma \in A$ into a proof of $P[\gamma]$.
\end{itemize}

The interactive proof system \texttt{Minlog} has been developed with the special intention of  extracting programs from mathematical proofs. The computational model consists of the theory of the partial continuous functionals in sense of Scott-Ershov 
and implements the theory of computable functionals \cite{Scott1982-domains}.
The atomic constructions are done by introduction of free algebras as types which are given by its constructors. 
For example, the free algebra for natural numbers is given by the constructors zero and successor, such that all other constructed elements  are also natural numbers.
Generally the types are built from the object types, type variables and type constants using algebra type formation and arrow type formation.
The later arrow type formation is denoted by ``$\rho\Rightarrow\sigma$'', where $\rho$ and $\sigma$ are two types.
The object types are usually types constructed by some free algebra variables. The type variables are specific variables of an unknown type. Consider for example a free algebra of lists. 
It is clear that the basic list manipulation like appending an element to some list or getting an element from the list should work independently of the type of the elements. 
Therefore the list manipulations are written in terms of type variables. 
The arrow type formation  $\rho\Rightarrow\sigma$ constructs a new type from the previously defined base types $\rho$ and $\sigma$, e.g., the successor of a natural number of type $\texttt{nat}$ is viewed as a function of type $\texttt{nat}\Rightarrow\texttt{nat}$.
The computation is implemented efficiently by the usage of the so-called normalization-by-evaluation. The reason of usage of typed theories is that the normalization-by-evaluation technique, which is applied to the proof terms, needs typed formulas.  As an example consider the construction of the type \texttt{nat} of natural numbers $\mathbb{N}$ by stating that \texttt{Zero} is of type \texttt{nat} and \texttt{Succ} is a mapping of type \texttt{nat$\Rightarrow$nat}. The normalization of the term 1+2 would result in \texttt{Succ(Succ(Succ Zero))} which is clearly of type $\texttt{nat}$ . The following types for the numbers with constructors and destructors are contained in the standard \texttt{Minlog} library:
\begin{itemize}
	\item \texttt{pos:One,SZero,SOne} (positive numbers $\mathbf{P}$ (binary))
	\item \texttt{nat:Zero,Succ} (natural numbers $\mathbf{N}$ )
	\item \texttt{list:Nil,Cons} 
	\item \texttt{int:IntP,IntZero,IntN} (integer numbers $\mathbf{Z}$)
	\item \texttt{rat:RatConstr} and destructors \texttt{RatN,RatD} (rational numbers $\mathbf{R}$)
	\item \texttt{rea:RealConstr} and destructors\\ \texttt{RealSeq,RealMod} (real numbers $\mathbf{R}$) 

\end{itemize} 
From this point onwards we use boldface symbols for numbers whenever we want to reason formally in \minlog.

 The  user-defined computation and corresponding rewriting rules are assigned to the so-called program constants. Usually the program constants are recursively written typed programs. Consider for example the multiplication of two natural numbers. The corresponding program constant \texttt{NatTimes} is of type \texttt{nat$\Rightarrow$nat$\Rightarrow$nat} and the computational rules are given recursively as:
\begin{align*}
&\texttt{NatTimes}\spc  n \spc \texttt{Zero} \rightarrow  \texttt{Zero} \\
&\texttt{NatTimes} \spc n \spc \left(\texttt{Succ} \spc m\right) \rightarrow \texttt{NatTimes}\spc n \spc (m + n).
\end{align*}
As an additional example consider the strict inequality relation $<$ again for the natural numbers. The program constant \texttt{NatLt} is of type \texttt{nat$\Rightarrow$nat$\Rightarrow$bool} and the computational rules are:
\begin{align*}
&\texttt{NatLt}\spc   n \spc \texttt{Zero} \rightarrow  \texttt{False} \\
&\texttt{NatLt} \spc 0 \spc (\texttt{Succ} \spc n) \rightarrow \texttt{True}\\
&\texttt{NatLt} \spc \left(\texttt{Succ} \spc n\right) \spc \left(\texttt{Succ} \spc m\right) \rightarrow \texttt{NatLt} \spc  n \spc m.
\end{align*}
Usually the computer proofs incorporate sometimes trivial goals and therefore must be proven automatically. Thus the so-called rewriting rules are important for the automation of theorem proving. For example assume that one has proved that $\forall a \in \mathbf{Q} \spc a+0=a$. Then one can add the machine-generated computation rule $a+0 \rightarrow a$ to the library of  rewriting rules. Every time the normalization process is done, if some rational $a+0$ occurs then  the zero is canceled automatically.
\subsection{The notion of totality}
In order to establish the connection between proofs and programs, each program constant in \texttt{Minlog} has to fulfill the totality requirement, i.e., the termination of the program constant under any correct input. For a comprehensive description of totality, we refer the interested reader to \cite[Chapter~8.3]{Stoltenberg1994-domain} and \cite{Berger1993-total}.   
 In order to use the whole machinery of the computational model in \texttt{Minlog}, program constants are required to be total functions. The totality is usually proved by induction directly after the definition of the program constant. For example assume that for some natural numbers $n,m$, the program constant \texttt{NatPlus n m} is recursively defined as:
\begin{itemize}
	\item \texttt{NatPlus n Zero $\rightarrow$ $n$} 
	\item \texttt{NatPlus n Succ m $\rightarrow$ Succ(NatPlus n m)}.
\end{itemize}
The totality goal is: $\forall \texttt{n,m} \spc \texttt{TotalNat n} \rightarrow \texttt{TotalNat m} \rightarrow \texttt{TotalNat NatPlus n m}$. By induction on $\texttt{m}$ and by definition $\phantom{v}\texttt{TotalNat(NatPlus n Zero)} \phantom{v}$ is equivalent to $\texttt{TotalNat(n)}$ which follows from the assumptions. The induction step 
is $\texttt{TotalNat(n +Succ m)}$ which is equivalent by definition to $\texttt{TotalNat(Succ(n +m)}$. Finally the proof tree is split into two goals $\texttt{TotalNat(Succ(n)}$ and $\texttt{TotalNat(n+m)}$. Clearly the successor of a natural number is total and the second goal holds due to the induction hypothesis.
In the present work, the totality proofs of all program constants have been shown and implemented.

\subsection{Inductive definitions and theorem proving}
Before one starts to prove theorems with the help of the computer, one has to provide  definitions which are relevant for the problem statement. In \texttt{Minlog} it is done by introducing the so-called inductively defined predicate constant (IDPC). Informally speaking, these are formations of implications leading to the proposition to be defined. As an example assume that one wants to provide a definition for rational numbers which lie in a unit interval. Then the corresponding IDPC \texttt{RatInUnitI} is given by:
\begin{equation*}
\forall a \in \mathbf{Q} \spc 0 \le a \rightarrow a \le 1 \rightarrow \texttt{RatInUnitI} \left[a\right]
\end{equation*}
which means semantically that to show $ \texttt{RatInUnitI} \spc a$ one has to show that $0 \le a$ and $a \le 1$. Additionally the definitions can also be provided inductively. As an example consider the definition of an even natural number given by the IDPC \texttt{Even}. Inductively one can proceed as follows:
\begin{align*}
&\texttt{Even} \left[\texttt{Zero}\right]\\
&\forall n \left( \texttt{Even}[n] \rightarrow \texttt{Even}[n+2]\right).
\end{align*}
Here the definition of the even natural number is given in terms of other natural number (in this case \texttt{Zero}) which is known to be even. In this work most of the IDPCs are given explicitly as in case of \texttt{RatInUnitI}.


\subsection{Preliminaries for the formalization in constructive setting}
This section is based on the foundation of  constructive analysis with witnesses \cite{Schwichtenberg2006-constructive}. The witnesses are algorithms that certify certain statements, which are propagated through the mathematical proofs. For example, informally, a witness for a  strictly increasing continuous function $f:I \rightarrow \mathbb{R}$  would be a map $\tau: \mathbb{Q} \rightarrow \mathbb{Q}$ such that for any $\varepsilon > 0$ and $a,b \in I$, one has  $ \spc 0<\tau(\varepsilon)<f(y)-f(x)$ whenever $0<\varepsilon<y-x$. By relying on the axiom of choice \cite[p.~89]{Ye2011-SF} or the dependent choice axiom \cite{Schwichtenberg2006-inverting}, a continuous inverse function $f^{-1}$ can be constructed in terms of the witness $\tau$.

All presented concepts are implemented within of \texttt{Minlog}, and   furthermore we extended the standard library of \texttt{Minlog}  by the necessary notions and theorems for vectors, continuous functions on multiple variables, vector-valued functions as well as auxiliary theorems for real numbers.
The  most important object in the current setting is real number. In \texttt{Minlog} the reals are built by introducing typed variables \texttt{as:$\mathbf{N}\Rightarrow\mathbf{Q}$} , \texttt{M:$\mathbf{P}\Rightarrow\mathbf{N}$}, defining the constructor by the arrow type formation \texttt{RealConstr: $(\mathbf{N}\Rightarrow\mathbf{Q}) \Rightarrow (\mathbf{P}\Rightarrow\mathbf{N})\Rightarrow\mathbf{R}$} and introducing the IDPC \texttt{Real}, i.e., a real $x$ is given by a regular Cauchy sequence of rationals $(a_n)_n$ (encoded by the typed variable \texttt{as}), denoted as \texttt{Cauchy},  with a given monotone modulus $M:\mathbf{P}\rightarrow\mathbf{N}$, denoted as \texttt{Mon}, that is $\forall p,m,n \spc |a_{n} - a_{m}|\le 2^{-p}$ whenever $M(p) \le n,m$ and $\forall p,q \spc p \le q \rightarrow M(p) \le M(q)$. Occasionally, we will also denote  a $2^{-p}$-approximation of the real $x=\left((a_n)_n,M\right)$ as $x^\lambda(p) $, which is defined as a rational $a_{M(p)}$.  In order to work with the IDPCs properly and realize natural deduction \cite{Schwichtenberg2009-program}, the so-called auxiliary elimination rules have to be proven, which roughly speaking are deconstructors of the IDPCs and may be considered as an analogue of an informal statement: "By the definition of  proposition $P$, proposition $Q$ holds". The next statement is an important example for such auxiliary elimination rules:     
\begin{elim}[\texttt{CauchyElim}] \label{thm:CauchyElim}
	\[
	\forall (a_n)_n,M \spc \mbox{\texttt{Cauchy}} \left[(a_n)_n, M \right]  \rightarrow \forall p,m,n \spc \left( M(p) \le n \rightarrow M(p) \le m  \rightarrow |a_{n} - a_{m}|\le 2^{-p} \right).
	\]
\end{elim}

A real $x=((a_n)_n,M)$ is called nonnegative if $\forall p \spc 0 \le a_{M(p)} + 2^{-p}$. For two real numbers $x$,$y$ denote $x \le_r y$ whenever the real $y-x$ is nonnegative. Similarly, a real is strongly nonnegative if $\forall n \spc 0 \le a_n$ and $x \le_s y$ whenever $y-x$ is strongly nonnegative.
For a real $x$ denote $x \in_p \mathbf{R}^+ $ if $2^{-p} \le a_{M(p+1)}$. We denote strict inequality of two reals $x,y$ by $x <_p y$, if there exists $p$ such that $y-x \in_p \mathbf{R}^+ $. Two reals $x=((a_n)_n,M)$ and $y=((b_n)_n,N)$ are equal ($x=_r y$), if 
$\forall p\in \mathbf{P}\spc|a_{M(p+1)}-b_{N(p+1)}|\le 2^{-p}$. Additionally, two reals are strongly equal, if $\forall n \spc a_n=b_n$, which is denoted by $x=_s y$. In some situations the strong versions of equality and inequality may ease the derivation of the corresponding non-strong requirements.  
The next three characterization theorems are often useful for operating with (in-)equalities within  our framework.

\begin{lem}[\texttt{RealEqChar}] \label{RealEqChar}
	Given two reals $x=((a_n)_n,M)$ and $y=((b_n)_n,N)$
	the following are equivalent:
	\begin{enumerate}
		\item $x=_r y$
		\item $\forall p \spc \exists n_0 \forall n \ge n_0 \spc \left(|a_n - b_n| \le 2^{-p}\right).$
	\end{enumerate}
\end{lem}
 
\begin{lem}[\texttt{RealNNegChar}]\label{RealNNegChar}
	For a real $x=((a_n)_n,M)$ the following are equivalent:
	\begin{enumerate}
		\item $0 \le_r x$
		\item $ \forall p \spc \exists n_0 \spc \forall n\ge n_0 \spc \left(-2^{-p}\le a_n\right)$.
	\end{enumerate}
\end{lem}
\begin{rem}
	The nonnegativity on real numbers is realized via the inductively defined predicate as opposed to the inequality on rational numbers, which is implemented as a simple program constant. Thus by the following IDPC:
	\[
	\texttt{ $\forall$x Real[x] $\rightarrow$  $\forall$ p 0 $\le$ x seq(x mod p)$+2^{-\texttt{p}} \rightarrow$ RealNNeg[x] } 
	\]
	in \minlog, where \texttt{seq} and \texttt{mod} are the display tokens for the deconstructors of the type \texttt{real}. Similarly, the IDPC for the equality is read as:
	\begin{align*}
	&\texttt{ $\forall$x,y Real[x] $\rightarrow$ Real[y] $\rightarrow$} \\ & \texttt{ $\forall$ p |x seq (x mod (p+1) - y seq (y mod (p+1))| $\le 2^{-\texttt{p}} \rightarrow$ RealEq[x,y]} .
	\end{align*}
	
\end{rem}
\begin{lem}[\texttt{RealPosChar}]\label{RealPosChar}
	For a real $x=((a_n)_n,M)$ the following are equivalent:
	\begin{enumerate}
		\item $\exists p \spc x \in_p \mathbf{R}^+$ 
		\item $\exists q,m \spc \forall n\ge m \spc \left(2^{-q}\le a_n\right)$.
	\end{enumerate}
\end{lem}

Additionally, the following important lemmata are provided to deal with the inequalities. Observe the necessity of the special shifts in the witness $p$ in the corresponding lemmata that involve strict inequalities.
\begin{lem}[\texttt{RealLeTrans}]
	$\forall \spc x,y,z \spc x \le_r y \rightarrow y \le_r z \rightarrow x \le_r z $.
\end{lem}


\begin{lem}[\texttt{RealPosCompat}]
	$\forall \spc x,y,p \spc x =_r y \rightarrow x \in_p \mathbf{R}^+ \rightarrow y \in_{p+2} \mathbf{R}^+ .$
\end{lem}




\begin{lem}[\texttt{RealLeLtTrans}] \label{RealLeLtTrans}
	$\forall \spc x,y,z,p \spc x \le_r y \rightarrow y <_p z \rightarrow x <_{p+3} z.$
\end{lem}
\begin{lem}[\texttt{RealNotGtIsLe}] \label{RealNotGtIsLe}
	$\forall \spc x,y \in \mathbf{R}\spc \lnot\left( \exists p \spc y <_p x \right) \rightarrow x \le_r y.$
\end{lem}





The next definitions introduce the concept of continuous functions. We will first introduce the definition due to \cite{Schwichtenberg2012-constr-analys} (\texttt{Cont}) and later define the concept of the completion monad \cite{Oconnor2008-certified} (\texttt{UCF}) in our framework.

\begin{defn}[\texttt{Cont}]\label{def:Cont}
	A uniformly continuous function $f:I \rightarrow \mathbf{R}$ on a compact interval $I$ with rational endpoints is given by
	\begin{enumerate}
		\item an approximation map $h_f :  (I \cap \mathbf{Q} ) \times \mathbf{N} \rightarrow \mathbf{Q} $ and a map $\alpha_f: \mathbf{P} \rightarrow \mathbf{N}$ such that  $(h_f(a,n))_n$ is a Cauchy sequence with modulus $\alpha_f$ for any $a \in (I \cap \mathbf{Q} )$
		\item a modulus  of continuity $\omega_f:\mathbf{P} \rightarrow \mathbf{P}$, which satisfies 
		\begin{equation*}
		|a-b| \le 2^{-\omega_f(p) +1} \rightarrow |h_f(a,n)-h_f(b,n)|\le 2^{-p}
		\end{equation*}		
	\end{enumerate}
	 for any $a,b \in (I \cap \mathbf{Q} )$ with monotone $\alpha_f$ and $\omega_f$.
\end{defn}
Note that a continuous function in sense of \texttt{Cont} attains only rational values for  efficiency reasons. Later on, the application procedure will allow us to generate real numbers from such continuous functions.
As an example, consider the representation of the function $f(x)=x^2$ on the interval $I = [0,2]$ in \texttt{Minlog}:\\
$\texttt{ContConstr}\left[0,2,\texttt{([a,n]a*a)}, \texttt{([p] 0)}, \texttt{([p] p+3)}\right]$
where \texttt{ContConstr} is a constructor of a function of type \texttt{cont} 
 which takes as the input the rational endpoints of the interval $I = [0,2]$, the approximating map \texttt{([a,n]a*a)} of type \texttt{$\mathbf{Q} \Rightarrow \mathbf{N} \Rightarrow \mathbf{Q} $}, the mapping $\alpha_f(p):=0$ of  type \texttt{$\mathbf{P} \Rightarrow \mathbf{N} $} as  modulus of convergence and the mapping $\omega_f(p):=p+3$ of type \texttt{$\mathbf{P} \Rightarrow \mathbf{P} $} as modulus of continuity. Along with the constructor, a proof of \texttt{Cont} has to be provided for this representation. That is, according to the definition of \texttt{Cont}
\begin{itemize}
	\item \texttt{([a,n]a*a)} is \texttt{Cauchy} with the modulus $\alpha_f(p):=0$ because the values of the sequence $(h_f(a_n,n))_n:=(a^2)_n$ are independent of $n$.
	\item The modulus of continuity $\omega_f(p):=p+3$ is correct and the corresponding goal  \[\forall a \in [0,2] \spc b \in [0,2] \spc p \in \mathbf{P} \spc \left( |a-b| \le 2^{-p-2} \rightarrow |a \cdot a-b \cdot b| \le 2^{-p}\right)\]
	holds because $|a \cdot a - b \cdot b| \le |(a-b)(a+b)| \le 2^{-p-2}\cdot|(a+b)| \le 2^{-p-2} \cdot 4 \le 2^{-p}$.\\
	Note that $(a+b) \le 4$ holds because   $a,b \in [0,2]$.
	\item $\alpha_f : p \mapsto 0$ and $\omega_f : p \mapsto p+3$ are both monotone.
\end{itemize}

 Note that the definition of the approximation map $h_f$ in \texttt{Cont} differs from the definition in \texttt{Minlog} in sense that the approximation maps in \texttt{Minlog} are defined as programs of type $\mathbf{Q}\Rightarrow \mathbf{N} \Rightarrow \mathbf{Q}$, omitting the information of the interval $I$. However this information is provided on the proof level by attaching the assumptions that the rational inputs have to belong to the specific interval $I$.

\begin{defn}[\texttt{Application}]\label{def:Application}
	The application of a uniformly continuous function $f:I\rightarrow \mathbf{R}$ to a real $x=((a_n)_n,M))$ in $I$, is defined to be a Cauchy sequence
	$(h_f(a_n,n))_n$ with modulus $\max(\alpha_f(p+2),M(\omega_f(p+1)-1)$ and is denoted by $f(x)$.
\end{defn}
\begin{lem}[\texttt{ApplicationCorr}]\label{thm:ApplicationCorr}
	For a uniformly continuous function $f:I\rightarrow \mathbf{R}$ and a real $x=((a_n)_n,M))$ in $I$, the Cauchy sequence emanating from \texttt{Application} and accompanied with the modulus $\max(\alpha_f(p+2),M(\omega_f(p+1)-1)$ is a real number.
\end{lem}

\begin{rem}
	The definition of the application incorporates the so-called extension theorem due to Bishop \cite[p.~91]{Bishop1985-constr-analysis}. In \texttt{Minlog} the application is realized by a program constant and allows one to address the lifting from rational mappings (see Definition \ref{def:Cont}) to the reals, which is at some point similar to the lifting technique of the completion monad for the constructive reals in \texttt{Coq} \cite{OConnor2007-monad}, which will be introduced later. Additionally, note that for the continuous function $f:I \rightarrow \mathbf{R}$ and real numbers $x,y \in I$ under the usage of the program \texttt{Application}, one is able to prove:
	\[
		|x-y| \le_r 2^{-\omega_f(p)} \rightarrow |f(x)-f(y)|\le_r 2^{-p} \,,
	\]
	which coincides to another usual existing notions of continuity in effective analysis \cite{Bishop1985-constr-analysis,Ko1991,Kushner1999-markov,Weihrauch2012-computable-analys}.
\end{rem}

\begin{lem}[\texttt{ApplicationCompat}]\label{thm:ApplicationCompat}
	For a uniformly continuous function $f:I\rightarrow \mathbf{R}$:
	\[
		\forall x,y \left(x=_r y \rightarrow f(x) =_r f(y) \right).
	\]
\end{lem}

To be able to operate on whole $\mathbf{R}$, we introduce the following IDPC for the global continuous functions:
\begin{idpc} [\texttt{Contex}] \label{idpc:contex}
	Consider the extended type \texttt{contex}:$(\mathbf{Q}\Rightarrow\mathbf{N}\Rightarrow \mathbf{Q}) \Rightarrow (\mathbf{Q}\Rightarrow\mathbf{Q}\Rightarrow\mathbf{P} \Rightarrow \mathbf{N}) \Rightarrow (\mathbf{Q}\Rightarrow\mathbf{Q}\Rightarrow\mathbf{P} \Rightarrow \mathbf{P})$, i.e., a triple $(h_{f_{ex}},\alpha_{f_{ex}},\omega_{f_{ex}})$ such that \[\forall [c,d] \spc \texttt{Cont} [c,d,h_{f_{ex}},\alpha_{f_{ex}}(c,d),\omega_{f_{ex}}(c,d)].\]
\end{idpc}

\begin{lem}[RealBound]\label{thm:RealBound}
	\begin{align*}
		\forall x (\texttt{Real} \left[x\right]  \rightarrow \exists a \in \mathbf{Q} \spc \forall n \spc |x \spc \texttt{seq} \spc n| \le a).
	\end{align*}
\end{lem}
The computational content of the Theorem \ref{thm:RealBound} is a program constant \texttt{cRealBound}.


%
%
%
%

As an alternative to the IDPC \ref{idpc:contex} we were also able to realize the concept of continuous function via the notion of the completion monad. Consider the maps $h_c:\mathbf{Q} \Rightarrow \mathbf{R}, \omega_{h_c}: \mathbf{P} \Rightarrow \mathbf{P}$ and the following IDPC:
\begin{idpc}[\texttt{UCF}]\label{def:ucf}
	\begin{align*}		
		&\forall h_c, \omega_{h_c} \left( \forall a\spc \texttt{Real} \left[h_c(a)\right] \rightarrow \forall a,b,p \left(|a - b | \le 2^{-\omega_{h_c}(p)} \rightarrow |h_c(a) - h_c(b)| \le_r 2^{-p}\right) \right)\\
		&\rightarrow \texttt{UCF} \left[\left(h_c,\omega_{h_c}\right)\right].
	\end{align*}
\end{idpc}
Now to realize the desired map of type $ \mathbf{R} \Rightarrow \mathbf{R}$ (such realization is sometimes referred as lifting) the tuple $f_c:=\left(h_c,\omega_{h_c}\right)$ is accompanied with the following program constant:
\begin{prog}[\texttt{UCFApplication}]\label{UCFapp}
	Consider the sequence of real numbers $(y_n)_n:=(h_c(x \spc \texttt{seq}\left(x \spc \texttt{mod}(\omega_{h_c}(n)+1)\right)))_n$ and assign the following real number:
	\begin{align*}
		f_c(x) \mapsto \texttt{RealConstr}\left[\left( y_{y_n \texttt{mod} \spc n + 2} \spc \texttt{seq} \spc n
		\spc  \right)_n ,(p+1)_p\right].
	\end{align*}	
\end{prog}
Similarly as in case of the Definition \ref{def:Application} we were able to show the correctness (\texttt{UCFApplicationReal}) and the compatibility with the equality (\texttt{UCFApplicationCompat}) respectively.
As an example of such construction consider the following realization of the approximation of the square roots, i.e., given some positive rational number $a$, $a_0=1$ and the sequence $a_{n+1}:= \frac{1}{2} \left(a_n + \frac{a}{a_n}\right)$ together with the computational content \texttt{cSqrtCauchyMod} of the proof of the following statement:
\begin{align*}
\texttt{SqrtCauchyMod:} \spc	\forall a \left( 0 < a \rightarrow \exists M \spc \texttt{Cauchy}[(a_n)_n,M]\right).
\end{align*}
We define the required maps
 \begin{align} 
&h_{\sqrt{\phantom{ }}}(a) := \begin{cases}
	\texttt{RealConstr} \spc \left((a_n)_n, \texttt{cSqrtCauchyMod} \spc a\right) , & \text{for } 0<a\\
	0, & \text{else}\label{fct:sqrt}\\
\end{cases}\\
&\omega_{\sqrt{\phantom{ }}}(p):= 2p\label{fct:sqrt2}
\end{align}
and prove that these maps indeed satisfy IDPC \ref{def:ucf} alongside with the actual properties of the square root e.g. $\forall x \left(0\le_r x \rightarrow x=_r\sqrt{x}\cdot \sqrt{x}\right)$.

The vectors, functions on multiple variables and vector-valued functions can be formalized in a straightforward way by utilizing the algebra of lists. The following representation allows to unfold the vector elements and their properties inductively.
\begin{defn}[Free algebra \texttt{list}]
		A list \texttt{xls} containing the elements of type $\alpha$ is either :
		\begin{itemize}
			\item \texttt{(Nil alpha)}, also called empty,
			\item constructed \texttt{xl::xls} with head \texttt{xl} and tail \texttt{xls}.
		\end{itemize}
\end{defn}
\begin{rem}
	A list contaning one element \texttt{xl} is denoted by $\texttt{xl:}$. For an arbitrary list \texttt{xls} we will denote the access to the i-th element of the list as $\texttt{xls}^i$.	
\end{rem}

\begin{defn}[\texttt{CauchyVector}]
		Consider the lists $\pmb v:\texttt{list }$$(\mathbf N \Rightarrow \mathbf Q) $, $M_{\pmb v}:\texttt{list }$$(\mathbf P \Rightarrow \mathbf N) $. A pair $\left(\pmb v,M_{\pmb v}\right)$ is \texttt{CauchyVector} if $|\pmb v|=|M_{\pmb v}|$ and either:
		\begin{itemize}
			\item $\pmb v$ and $M_{\pmb v}$ are empty,
			\item the pair (\texttt{head $\pmb v$}, \texttt{head $M_v$}) is \texttt{Cauchy} and the pair (\texttt{tail $\pmb v$}, \texttt{tail $M_{\pmb v}$}) is \texttt{CauchyVector}.
		\end{itemize}
\end{defn}

\begin{defn}[\texttt{MonVector}]
	A list $M_{\pmb v}:\texttt{list }$$(\mathbf P \Rightarrow \mathbf N) $ is  \texttt{MonVector} if either:
	\begin{itemize}
		\item $M_{\pmb v}$ are empty,
		\item the map \texttt{head $M_{\pmb v}$} is \texttt{Mon} and the list \texttt{tail $M_{\pmb v}$} is \texttt{MonVector}.
	\end{itemize}
\end{defn}

\begin{defn}[\texttt{RealVector}]
	For the lists $\pmb v:\texttt{list }$$(\mathbf N \Rightarrow \mathbf Q) $, $M_{\pmb v}:\texttt{list }$$(\mathbf P \Rightarrow \mathbf N) $ of length $n$,  a vector $\pmb x=(\pmb v,M_{\pmb v}) \in \mathbf{R}^n$ if the pair $(\pmb v,M_{\pmb v})$ is \texttt{CauchyVector} and $M_{\pmb v}$ is \texttt{MonVector}.
\end{defn}For a vector $\pmb x=(\pmb v,M_{\pmb v}) \in \mathbf{R}^n$ the quantity $\pmb x \spc \texttt{dim}:= |v|$ $\left( \text{if } |\pmb v|=|M_{\pmb v}|\right)$ gives the number of the components of the vector. 
Two vectors $\pmb x, \pmb y$ are called compatible if both are \texttt{RealVector} and $\pmb x \spc \texttt{dim}=\pmb y \spc \texttt{dim}$. Two vectors $\pmb x,\pmb y$ are called equal ($\pmb x=_r \pmb y$) if they are compatible and have equal components. 
We provide the reflexivity, transitivity and symmetry theorems for this relation. 
%

Similarly to Definition \ref{def:Cont}, a function on multiple variables can be defined as follows:
\begin{defn}[\texttt{Contmv}]\label{def:Contmv}
	A uniformly continuous function of multiple variables $g:B \rightarrow \mathbf{R}$ on a compact ball $B=(\pmb e_0,R) \subset \mathbf{R}^n$ with rational center $\pmb{e}_0$ and rational radius $R$ is given by
	\begin{enumerate}
		\item an approximation map $h_{g}:(B \cap \mathbf{Q}^n ) \times \mathbf{N} \rightarrow \mathbf{Q} $ and a map $\alpha_{g}: \mathbf{P} \rightarrow \mathbf{N}$ such that $(h_{g}(\pmb e,n))_n$ is a Cauchy sequence with modulus $\alpha_g$ for any $\pmb e \in (B \cap \mathbf{Q}^n )$
		\item a modulus $\omega_{g}:\mathbf{P} \rightarrow \mathbf{P}$ of continuity, which satisfies 
		\begin{equation*}
		\|\pmb e_1- \pmb e_2\|_\infty \le 2^{-\omega_{g}(p) +1} \rightarrow |h_{g}(\pmb e_1,n)-h_{g}(\pmb e_2,n)|\le 2^{p}
		\end{equation*}		
	\end{enumerate}
	for any $\pmb e_1,\pmb e_2 \in (B \cap \mathbf{Q}^n )$ with monotone $\alpha_{g}$ and $\omega_{g}$.
\end{defn}
\begin{rem}
	The choice of the norm $\| \cdot \|$ influences the values of moduli $\alpha_g$ and $\omega_g$. Usually, one would prefer to use either $1$-norm or $\infty$-norm for  computational purpose, since for rational vectors the values of the norm are also rationals, which is not the case, e.g., for the Euclidean norm.
\end{rem}


%
From here there are multiple ways to define vector-valued functions that act on full $\mathbf{R}^n$. 
In order to ease the derivation, we decided to implement maps as a type $\pmb f: \mathbf{R}^n \Rightarrow \mathbf{R}^m$ that are compatible with the equality,i.e., 
\begin{defn}[\texttt{ProperMap}]\label{Cmap}
	Given the map $\pmb f: \mathbf{R}^n \Rightarrow \mathbf{R}^m$, we call the tuple $(\pmb f,n,m)$ proper if
	\begin{enumerate}
		\item $\forall \spc \pmb x,n,m \left(\pmb x \in \mathbf{R}^n \rightarrow \pmb f(\pmb x) \in \mathbf{R}^m \right)$
		\item $\forall \spc \pmb x,\pmb y \left( \pmb x =_r \pmb y \rightarrow   \pmb f(x) =_r \pmb f(\pmb x)  \right)$.
	\end{enumerate}
\end{defn}
The first property guarantees that the dimensions and the approximations are provided correctly. The second property is the compatibility of the map $\pmb{f}$ with the equality (in sense of vectors). Due to the fact that all continuous functions as well as the operations $+,-,\cdot,/$ on real numbers are compatible with the equality, it allows to parse function terms in a more convenient way without the necessity of deriving appropriate moduli of  continuity for some arbitrary vector-valued continuous function fully. Furthermore, the terms may origin from different sources, i.e., given by the Definition \ref{def:Cont} (\texttt{Cont}) or by the Definition \ref{def:ucf} (\texttt{UCF}) and can be combined together. This definition may also be seen as an abstract interface and reduces the difficulty of interactive theorem proving drastically.

As a certain non-trivial example consider the proof of compatibility with the equality of the functions of type $\mathbf{R}^1 \Rightarrow \mathbf{R}^1$ defined by the completion monad (\texttt{UCF}).
\begin{thm}[\texttt{UCFApplicationCompat}]
\begin{align}
	\forall f_c \left(\texttt{UCF} \spc [f_c] \rightarrow x=_r y \rightarrow f_c(x) =_r f_c(y)  \right)
\end{align}	 
\end{thm}
\begin{proof}
	Assumming that $f_c=(h_c,\omega_c)$, $x=((a_n)_n,M)$, $y=((b_n)_n,N)$ and $x=_r y$, we define following real numbers:
	\begin{itemize}
		\item $z_1 := h_c(a_{M(\omega(n)+1)})$
		\item $z_2 := h_c(b_{N(\omega(n)+1)})$
	\end{itemize}
	Then the goal $f_c(x) =_r f_c(y)$   
	is equivalent by Lemma \ref{RealEqChar} and by means of \texttt{UCFApplication} \ref{UCFapp} to the goal:
	\begin{align}
		\forall p \spc \exists n_0 \spc \forall n \ge n_0 \left(\left|z_1^{\lambda}(n+2) -  z_2^{\lambda}(n+2)\right| \le 2^{-p}\right)
	\end{align}
	Now for any $p$ we set $n_0=p+3$ and for any $n \ge n0$ let $z_3:=$ and $z_4:=$. Then
	\begin{align*}
		\left|z_1^{\lambda}(n+2) -  z_2^{\lambda}(n+2)\right| &\le \left|z_1^{\lambda}(n+2) -  z_3\right| + \left|z_3 -  z_4\right| + \left|z_4 -  z_2^{\lambda}(n+2)\right|\\
		& \le 2^{-p-2} + 2^{-p-1} + 2^{-p-2}=2^{-p}.  
	\end{align*}
	where 
	$\left|z_1^{\lambda}(n+2) -  z_3\right|,\left|z_4 -  z_2^{\lambda}(n+2)\right| \le 2^{-p-2}$ by definition of the Cauchy sequence.
	For the middle term consider the equivalent rewritten goal:
	\begin{align*}
		-2^{-p-2} \le 2^{-p-2} + - \left|z_3 -  z_4\right| = 2^{-p-2} - (|z_1-z_2|)^{\lambda}(n+2) 
	\end{align*}
	which is in turn equivalent by Lemma \ref{RealNNegChar} together with the fact that $p+3 \le n$:
	\[
		 |z_1-z_2|=\left|h_c(a_{M(\omega(n)+1)}) - h_c(b_{N(\omega(n)+1)})\right|\le_r 2^{-p-2} 
	\]
	which holds by Definition \ref{UCFapp} and the fact that \[
	\left|x^\lambda(\omega(n)+1) - y^\lambda(\omega(n)+1)\right| \le 2^{-\omega(n)} \le 2^{-\omega(p+2)}
	\]
	which follows from Lemma \ref{RealEqChar}, condition that $x =_r y$ and since $p+2 \le n$. 
\end{proof}
\begin{rem}
	The above proof is a partial proof of the statement that the map $x \mapsto \texttt{UFCApplication} \spc f_c \spc x$ is proper in sense of Definition \ref{Cmap}.
\end{rem}

In summary, in this section we presented the most important concepts of the \texttt{Minlog} system and presented the formalization of constructive analysis in \texttt{Minlog} as well as the extension of the standard library which was necessary for our goal. In the next section we provide the proof-theoretical verification of the SOS Lyapunov functions in our framework.
\section{Formal Proofs of Stability for Polynomial Systems}
\label{sec:formal}
Once the SOS Lyapunov function has been obtained by  Algorithm \ref{alg:exactLyapunov}, the proof of the final formal verification goal involves dealing with nonlinear inequalities. In this section we make a brief remark on the general undecidability of the problem as well as provide technical details about the implementation of polynomials in \texttt{Minlog}. Finally we present the developed framework for the formal verification of Lyapunov functions by  introducing various inductive definitions for positive definite and nonnegative functions and by establishing the proof theoretical connections between these definitions.

Given an autonomous discrete-time difference equation \begin{equation}\label{sys:dt}
\pmb x^+=\pmb F(\pmb x),\pmb x(0)=\pmb x_0 \,
\end{equation}
the solution of \eqref{sys:dt} is a sequence of real numbers arising from recursive application of continuous vector-valued function $\pmb F$ starting at $\pmb x_0\in \mathbf{R}^m$. 
In order to ease the derivation, we will consider the function \pmb{F} as some mapping given by the type $\left(\mathbf{R}^m \Rightarrow \mathbf{R}^m\right)$ which satisfies the following IDPC:
\begin{idpc}[\texttt{ProperRHS}]
	\begin{align*}
		\texttt{ProperRHS: } &\texttt{$\forall \spc \pmb{F},m  \spc ( \forall \pmb x \left( \pmb x \spc \text{dim} = m \rightarrow x \spc \text{dim} = (\pmb F \spc \pmb x) \text{dim}  \right)$}\rightarrow \\
		&\texttt{$\forall \pmb x \left( \pmb x \spc \text{dim} = m \rightarrow \text{RealVector} \spc \pmb x \rightarrow \texttt{RealVector} \spc (\pmb F \spc \pmb x) \right)$ $\rightarrow$} \\
		& \texttt{$\forall \pmb x,\pmb y
			\left( \pmb x \spc \text{dim} = m \rightarrow \pmb x =_r \pmb y \rightarrow (\pmb F \spc \pmb x)=_r(\pmb F \spc \pmb y)\right) \rightarrow$
		}\texttt{ProperRHS}\left[\pmb F,m \right]).
	\end{align*}
\end{idpc}
These three requirements on $\pmb{F}$ are sufficient to prove all relevant intermediate goals which will arise during our interactive theorem proving. 

We introduce the sequences of vectors  $\left(\pmb x(n)\right)_n$ as type ($\mathbf{N} \Rightarrow \mathbf{R^m})$  and 
construct the sequences via application of $\pmb F$ recursively and hence define the solution of \eqref{sys:dt} by comparing the given real vector sequence with the constructed sequence, as it is shown via following program constant and the corresponding IDPC. 
\begin{prog}[\texttt{Simulation}] \label{def:Simulation}
	\phantom{text}
	\begin{align*}
		 &\texttt{Simulation}\left[\texttt{Zero},\pmb x_0,\pmb F\right] \mapsto \pmb x_0 \\
		 &\texttt{Simulation} \left[(\texttt{Succ } n) , \pmb x_0, \pmb F \right]\mapsto \pmb F \left( \texttt{Simulation} \left[ n , \pmb x_0 , \pmb F\right]\right).
	\end{align*}
\end{prog}
\begin{idpc}[\texttt{Solution}] \label{def:Solution}
	The sequence of real numbers arising from recursive application of $\pmb F$, which satisfies the IDPC \texttt{ProperRHS} starting at $\pmb x_0$ is called solution of the system \ref{sys:dt}, i.e.,
	\begin{align*}
		\forall \left(\pmb{x}(n)\right)_n, \pmb F, \pmb x_0 \left( \forall n \left( \pmb{x}(n)=_r\texttt{Simulation} \left[ n , \pmb x_0 , \pmb F\right] \right) \rightarrow \texttt{Solution}\left[\left(\pmb{x}(n)\right)_n , \pmb F , \pmb x_0 \right]\right).
	\end{align*}
\end{idpc}
Let $\left(\pmb{x}(n)\right)_n$ be such solution.
We may now define multiple versions of stability of the equilibrium solution in the following way:
\begin{defn}[\texttt{StableEqLe}] \label{def:StableEqLe}
	The equilibrium is called stable if for any $\varepsilon>0$ there exists $\delta$ such that  $\|\pmb x(0)\|\le_r\delta$ implies $\|\pmb x(n)\| \le_r \varepsilon$ for all $n \ge 0$. 
\end{defn}

\begin{idpc}[\texttt{StableEq}] \label{idpc:Stable}
	\begin{align*}
		&\forall \left(\pmb{x}(n)\right)_n, \pmb F, \pmb x_0, \| \cdot \| \left( \texttt{Solution}\left[\left(\pmb{x}(n)\right)_n , \pmb F , \pmb x_0 \right] \rightarrow \texttt{Norm} \left[\| \cdot \| \right] \rightarrow \right. \\&\forall \varepsilon > 0 \spc \left. \exists \delta \left(\| \pmb x(0) \|\le_r \delta \rightarrow \forall n \left( \| \pmb x(n)\| \le_r \varepsilon\right)\right) \rightarrow \texttt{StableEq} \left[\left(\pmb{x}(n)\right)_n,\pmb F,\pmb x_0, \| \cdot \|\right]  \right).
	\end{align*}
\end{idpc}

\begin{defn}[\texttt{StableEqLt}] \label{def:StableEqLt}
	The equilibrium is called stable if for any $\varepsilon>0$ there exists $\delta$ such that  $\|\pmb x(0)\|<\delta$ implies $\|\pmb x(n)\| < \varepsilon$ for all $n \ge 0$. 
\end{defn}
\begin{idpc}[\texttt{StableEqLt}] \label{idpc:StableEqLt}
	\begin{align*}
	&\forall \left(\pmb{x}(n)\right)_n, \pmb F, \pmb x_0, \| \cdot \| \left( \texttt{Solution}\left[\left(\pmb{x}(n)\right)_n , \pmb F , \pmb x_0 \right] \rightarrow \texttt{Norm} \left[\| \cdot \| \right] \rightarrow \right. \\&\forall \varepsilon > 0 \spc \left. \exists \delta \left( \exists p \spc \| \pmb x(0) \|<_p\delta \rightarrow \forall n \left( \exists p \spc  \| \pmb x(n)\| <_p \varepsilon\right)\right) \rightarrow \texttt{StableEqLt} \left[\left(\pmb{x}(n)\right)_n,\pmb F,\pmb x_0, \| \cdot \|\right]  \right).
	\end{align*}
\end{idpc}

\begin{defn}[\texttt{StableEqLtClass}] \label{def:StableEqLtClass}
	The equilibrium is called stable if for any $\varepsilon>0$ there exists $\delta$ such that  $\exists_{cl} \spc p \spc \|\pmb x(0)\|<_p \delta$ implies $\exists_{cl} \spc p \spc \|\pmb x(n)\| <_p \varepsilon$ for all $n \ge 0$. 
\end{defn}
\begin{idpc}[\texttt{StableEqLtClass}] \label{idpc:StableEqLtClass}
	\begin{align*}
	&\forall \left(\pmb{x}(n)\right)_n, \pmb F, \pmb x_0, \| \cdot \| \left( \texttt{Solution}\left[\left(\pmb{x}(n)\right)_n , \pmb F , \pmb x_0 \right] \rightarrow \texttt{Norm} \left[\| \cdot \| \right] \rightarrow \right. \\&\forall \varepsilon > 0 \spc \left. \exists \delta \left( \exists_{cl} \spc p \spc \| \pmb x(0) \|<_p\delta \rightarrow \forall n \left( \exists_{cl} \spc p \spc  \| \pmb x(n)\| <_p \varepsilon\right)\right) \rightarrow \texttt{StableEqLt} \left[\left(\pmb{x}(n)\right)_n,\pmb F,\pmb x_0, \| \cdot \|\right]  \right).
	\end{align*}
\end{idpc}



We define a special continuous function, which is (historically) so-called ``function of class-$\mathcal{K}$''. 

\begin{idpc}[\texttt{ClassKappa}]\label{idpc:kappa}
	\begin{align*}
&\forall \kappa  \spc \texttt{Contex} [\kappa]\rightarrow  \forall x \left(x=_r 0 \rightarrow \kappa(x)=_r 0\right) \rightarrow \\
&\forall x,y \left(0\le_r x \rightarrow 0 \le_r y \rightarrow x \le_r y \rightarrow \kappa(x) \le_r \kappa(y)\right)\rightarrow\\
&\forall x,y \left(0\le_r x \rightarrow 0 \le_r y\rightarrow \kappa(x) \le_r \kappa(y)  \rightarrow x \le_r y \right)\rightarrow\\
&\forall x \left( \exists p \spc 0 <_p x \rightarrow \exists p \spc 0 <_p \kappa(x) \right)\rightarrow \texttt{ClassKappa}\left[\kappa\right]
\end{align*}
\end{idpc}

\begin{idpc}[\texttt{ClassKappaPos}]\label{idpc:kappapos}
	\begin{align*}
	&\forall \kappa  \spc \texttt{ClassKappa} [\kappa]\rightarrow  \\
	&\forall x,y \left( 0\le_r x \rightarrow  \exists p \spc x <_p y \rightarrow \exists p \spc \kappa(x)  <_p \kappa(y) \right)\rightarrow  \\
	&\forall x,y \left(0\le_r x \rightarrow  \exists p \spc \kappa(x)  <_p \kappa(y)\rightarrow \exists p \spc x <_p y \right)\rightarrow \texttt{ClassKappaPos}\left[\kappa\right]. \\
	\end{align*}
\end{idpc}
Function, having the property \ref{idpc:kappapos} will be also denoted as a function of class-$\mathcal{K}_{<}$. Additionally, for some functions of class-$\mathcal{K},\mathcal{K}_{<}$, we require the existence of continuous proper inverse.

\begin{idpc}[\texttt{ClassKappaInverse}]\label{idpc:kappainv}
\begin{align*}
	&\forall \kappa_1,\kappa_1^{-1} \texttt{ClassKappa} \left[\kappa_1\right] \rightarrow \texttt{ProperMap}[\kappa_1^{-1},1,1] \rightarrow\\
	&\forall x \left(0\le_r x \rightarrow \kappa_1^{-1}(\kappa_1(x))=_r x\right) \rightarrow \\
	&\forall y \left(0\le_r y \rightarrow \kappa(\kappa_1^{-1})=_ry\right) \rightarrow \\
	&\forall x \left(0 \le_r x \rightarrow 0 \le_r \kappa^{-1}(x)\right) \rightarrow \\
	&\forall x \left( \exists p \spc 0 <_p x \rightarrow \exists p \spc 0 <_p  \kappa^{-1}(x) \right) \rightarrow \texttt{ClassKappaInverse}[\kappa,\kappa^{-1}] 
\end{align*}
\end{idpc}

\begin{rem}
	Using original classical definitions, it is possible to show. that for a function $\kappa$ of class-$\mathcal{K}$, the inverse $\kappa^{-1}$ is also a function of class-$\mathcal{K}$. In constructive setting, however, as opposed to some classically obvious facts, the reasoning is much harder and requires more advanced techniques \cite{mandelker1982continuity,mandelkern1983constructive}, that are not implemented to the authors best knowledge yet. However, some ideas in \cite[Proposition 2.8,Proposition 2.12]{mandelkern1983constructive} might be fruitful for further development of corresponding constructive theorems in proof assistants. 
\end{rem}

\begin{prop}[\texttt{LyapunovStableLe}]\label{LyapunovStableLe}
	Assume there exists a function $V(\pmb x)$ and a class-$\mathcal{K}_<$-function $\kappa_1$ such that:
	\begin{align}
		&\kappa_1(\|\pmb x\|) \le_r V(\pmb x) &  \\
		&V(\pmb x(n+1)) \le_r V(\pmb x(n)) &\spc (\texttt{PDVDiff}) \ref{pdvdiffelim}
	\end{align}
	for any solution $\left(\pmb x(n)\right)_n$ of \eqref{sys:dt}.
	Then the equilibrium of \eqref{sys:dt} is stable in sense of Definition \ref{def:StableEqLe}. 
\end{prop}
\begin{proof}
	From $\kappa_1(\|\pmb x\|) \le_r V(x)$ we first obtain the following assertions:
	\begin{enumerate}
		\item $\exists p \spc 0 <_p \|\pmb x\| \rightarrow \exists \spc p \spc  0 <_p V(\pmb x)$
		\item $\forall \theta_1 \spc \exists \theta_2 \left( \exists p \spc \theta_1 <_p \|\pmb x\| \rightarrow \theta_2 \le_r V(\pmb x)\right) $
	\end{enumerate}
	that imply that for any $\varepsilon$ we can construct a $\theta' \in (0,\theta_2)$.
	Additionally, by providing the appropriate modulus of continuity (w.r.t. the specified vector norm) for $V$, we have 
	\[ \exists \spc  p \spc 0 <_p \theta_2 \rightarrow \exists \delta \|\pmb x\| \le_r \delta \rightarrow V(\pmb x) \le_r \theta' .\] 
	Let $\delta_0$ be the quantity such that $V(\pmb x) \le_r \theta_2$ whenever $\|\pmb x\| \le_r \delta_0$. We claim that $\forall n (\|\pmb x_n\|\le_r \varepsilon)$ whenever $\|\pmb x_0\| \le_r \delta_0$.
	For some $n$ the statement $\|x_n\|\le_r \varepsilon$ is equivalent to $\neg (\exists p \spc \varepsilon <_p \|\pmb x_n\|)$ by Lemma \ref{RealNotGtIsLe} which in turn is equivalent to $ (\exists p \spc \varepsilon <_p \|\pmb x_n\|) \rightarrow \texttt{False}$. Additionally, we have $(\exists p \spc \theta' <_p V(\pmb x_n)) \rightarrow \texttt{False}$, since $V(\pmb x_n) \le_r V(\pmb x_0) \le_r \theta'$ by the condition \texttt{PDVDiff}, which means that proving \[\exists p \spc \theta' <_p V(\pmb x_n)\] would complete the proof. Since $\theta' \in (0,\theta_2)$, we have $p_0$ such that $\theta' <_{p_0} \theta_2$, which in turn means that $\theta' <_{p_0+3} V(\pmb x_n)$ by Lemma \ref{RealLeLtTrans} since $\theta_2 \le_r V(\pmb x_n)$.
\end{proof}

\begin{rem} \label{PosChooseProperAux}
	The appropriate choice of $\theta' \in (0,\theta)$ can be derived as follows.
	Assume that the real $\theta$ is given by the tuple $\left((a_n)_n,M\right)$. Let $p_0$ be the proof that $0<_{p_0} \theta$, i.e. $2^{-p_0} \le a_{M(p_0+1)}$. Introduce $\theta'$ as a rational number $a_{M(p_0+4)} - 2^{-p_0-2}$. Clearly $\theta'$ is a real number. Furthermore $\theta'<_{p_0+2} \theta$, which follows directly from the definition of the strict inequality. Finally, $0 <_{p0+2} \theta'$ due to the proper choice of the approximation parameter $p_0$ and the fact that the tuple $\left((a_n)_n,M\right)$ is \texttt{Cauchy}.
\end{rem}

\begin{prop}[\texttt{LyapunovStableLt}]\label{LyapunovStableLt}
	Assume there exists a function $V(\pmb x)$ and a class-$\mathcal{K}_<$-function $\kappa_1$ such that:
	\begin{align}
	&\kappa_1(\|\pmb x\|) \le_r V(\pmb x) &  \\
	&V(\pmb x(n+1)) \le_r V(\pmb x(n)) &\spc (\texttt{PDVDiff}) \label{pdvdiffelim}
	\end{align}
	for any solution $\left(\pmb x(n)\right)_n$ of \eqref{sys:dt}.
	Then the equilibrium of \eqref{sys:dt} is stable in sense of Definition \ref{def:StableEqLt}.
\end{prop}
\begin{proof}
	  Let $\varepsilon>0$ and choose $\varepsilon' \in (0,\varepsilon)$ as in Remark \ref{PosChooseProperAux}, i.e. $\exists p \spc 0 <_p \varepsilon' \land \exists p \spc \varepsilon' <_p \varepsilon$. The conditions of the proposition imply that the equilibrium of \eqref{sys:dt} is stable in sense of Definition \ref{def:StableEqLe} by Proposition \ref{LyapunovStableLe}, i.e. $\exists \delta \spc \|\pmb x_0\| \le_r \delta \rightarrow \forall n \spc \|\pmb x_n\|\le_r \varepsilon'$. Let $\delta_0$ be such quantity. The goal is read as:
	  \[
	  \exists p \spc \|\pmb x_0\| <_p \delta_0 \rightarrow \forall n \spc \exists p \spc \|\pmb x_n\| <_p \varepsilon.
	  \]
	  Let $p_0$ be the quantity such that $\varepsilon' <_{p_0} \varepsilon$. Then for any $n$ under consideration of the effect of Lemma \ref{RealLeLtTrans}, 
	  \begin{align*}
	  	\exists p \spc \|\pmb x_0\| <_p \delta_0\rightarrow \|\pmb x_0\| \le_r \delta \rightarrow \|\pmb x_n\|\le_r \varepsilon' \rightarrow \varepsilon' <_{p_0} \varepsilon \rightarrow \|\pmb x_n\| <_{p_0+3} \varepsilon.
	  \end{align*}
\end{proof}

\begin{prop}[\texttt{LyapunovStableLtClass}]\label{LyapunovStableLtClass}
	Assume there exists a function $V(\pmb x)$ and a class-$\mathcal{K}_<$-functions $\kappa_1$ such that:
	\begin{align}
	&\kappa_1(\|\pmb x\|) \le_r V(\pmb x) &  \\
	&V(\pmb x(n+1)) \le_r V(\pmb x(n)) &\spc (\texttt{PDVDiff}) \ref{pdvdiffelim}
	\end{align}
	for any solution $\left(\pmb x(n)\right)_n$ of \eqref{sys:dt}.
	Then the equilibrium of \eqref{sys:dt} is stable. 
\end{prop}

\begin{proof}
	From $\kappa_1(\|\pmb x\|) \le_r V(x)$ we first obtain the following assertions:
	\begin{enumerate}
		\item $\exists p \spc 0 <_p \|\pmb x\| \rightarrow \exists \spc p \spc  0 <_p V(\pmb x)$
		\item $\forall \theta_1 \in \mathbf{R} \spc \exists \theta_2 \left( \exists p \spc \theta_1 \le_r \|\pmb x\| \rightarrow \theta_2 \le_r V(\pmb x)\right) $
	\end{enumerate}
	that imply that for any $\varepsilon$ we can construct a $\theta' \in (0,\theta_2)$.
	Additionally, by providing the appropriate modulus of continuity for $V$, we have 
	\[ \exists \spc  p \spc 0 <_p \theta_2 \rightarrow \exists \delta \|\pmb x\| \le_r \delta \rightarrow V(\pmb x) \le_r \theta' .\] 
	Let $\delta_0$ be the quantity such that $V(\pmb x) \le_r \theta_2$ whenever $\|\pmb x\| \le_r \delta_0$. We claim that $\forall n ( \exists_{cl}\spc p  \|\pmb x_n\|<_p \varepsilon)$ whenever $\exists p \|\pmb x_0\| <_p \delta_0$.
	To see this, observe, that for some fixed $n$ the statement $ \exists_{cl} \spc p \|\pmb x_n\|<_p \varepsilon$ is equivalent to  $ (\varepsilon \le_r \|\pmb x_n\|) \rightarrow \texttt{False}$. Additionally, we have $(\exists p \spc \theta' <_p V(\pmb x_n)) \rightarrow \texttt{False}$, since $V(\pmb x_n) \le_r V(\pmb x_0) \le_r \theta'$, which means that proving \[\exists p \spc \theta' <_p V(\pmb x_n), \text{ whenever }\varepsilon \le \|\pmb x_n\|,\] would complete the proof. Since $\theta' \in (0,\theta_2)$, we have $p_0$ such that $\theta' <_{p_0} \theta_2$, which in turn means that $\theta' <_{p_0+3} V(\pmb x_n)$ by Lemma \ref{RealLeLtTrans} and since $\theta_2 \le_r V(\pmb x_n)$.
\end{proof}

\begin{rem}
	The original classical proof requires the construction of the  set $\Omega_{\varepsilon}:=\{\pmb x \in B_{\varepsilon} | V(\pmb x) \le \theta'\}$ and reasoning that $\Omega_{\varepsilon}$ lies in interior of $B_{\varepsilon}$, is compact and invariant for $\pmb F$, that might be complicated in the constructive setting. Our versions use simple  manipulations on inequalities.
\end{rem}
The following proposition allows one to state Lyapunov uniform stability. The proof is based on \cite[Theorem 2.16]{Halanay2000-stability}.
\begin{prop}[\texttt{LyapunovUniformStableLe}]\label{LyapunovUniformStableLe}
	Assume there exists a function $V(x)$ and two class-$\mathcal{K}$-functions $\kappa_1,\kappa_2$ such that:
	\begin{align}
&\kappa_1(\|\pmb x\|) \le_r V(\pmb x) \le_r \kappa_2(\|\pmb x\|)&\spc (\texttt{PDV})  \\ 
&V(\pmb x(n+1)) \le_r V(\pmb x(n)) &\spc (\texttt{PDVDiff}) \ref{pdvdiffelim}
\end{align}
	for any solution $\left(\pmb x(n)\right)_n$ of \eqref{sys:dt}.
	Then the equilibrium is stable in sense of Definition \ref{def:StableEqLe}. 
\end{prop}
\begin{proof}
	The Definition \ref{def:StableEqLe} (\texttt{StableLe})
	requires to introduce for any $\varepsilon>0$ the quantity $\delta$ such that  whenever $\|\pmb x(0)\| \le_r \delta$ we have $\forall n \spc \|\pmb x(n)\| \le_r \varepsilon$. We claim that $\delta  =\kappa_2^{-1}(\kappa_1(\varepsilon))$.
	From condition (\ref{pdvdiffelim}) (\texttt{PDVDiff}) we obtain $\forall  n \spc V (\pmb x(n))) \le_r V (\pmb x_0)$.
	Now, by the properties of the IDPCs \ref{idpc:kappa} and \ref{idpc:kappainv} we are able to show that:
	\begin{align*}
	\kappa_1(\pmb x(n)) \le_r V(\pmb x(n)) \le_r V(\pmb x_0) \le_r \kappa_2(\|\pmb x_0\|) \le_r \kappa_2(\delta)\le_r \kappa_2(\kappa_2^{-1}(\kappa_1(\varepsilon))) \le_r\kappa_1(\varepsilon) \,.
	\end{align*}
	And thus, $\kappa_1( \|\pmb x(n) \| ) \le_r \kappa_1(\varepsilon)  \rightarrow \|\pmb x(n)\| \le_r \varepsilon$.
\end{proof}

\begin{prop}[\texttt{LyapunovUniformStableLt}]\label{LyapunovUniformStableLt}
	Assume there exists a function $V(x)$ and two class-$\mathcal{K}_{<}$-functions $\kappa_1,\kappa_2$ such that:
	\begin{align}
	&\kappa_1(\|\pmb x\|) \le_r V(\pmb x) \le_r \kappa_2(\|\pmb x\|)&\spc (\texttt{PDV})  \\ 
	&V(\pmb x(n+1)) \le_r V(\pmb x(n)) &\spc (\texttt{PDVDiff}) \ref{pdvdiffelim}
	\end{align}
	for any solution $\left(\pmb x(n)\right)_n$ of \eqref{sys:dt}.
	Then the equilibrium is stable in sense of Definition \ref{def:StableEqLt}. 
\end{prop}
\begin{proof}
	The Definition \ref{def:StableEqLt} (\texttt{StableLt})
	requires to introduce for any $\varepsilon>0$ the quantity $\delta$ such that whenever $\exists p \spc \|\pmb x(0)\| <_p \delta$ we have $\forall n \spc \exists p \spc \|\pmb x(n)\| <_p \varepsilon$.  We claim that $\delta  =\kappa_2^{-1}(\kappa_1(\varepsilon))$ and let $p_0$ be the quantity such that $\kappa_2(\pmb x(0))<_{p_0} \kappa_2(\kappa_2^{-1}(\kappa_1(\varepsilon)))$, which is derivable from the assumption that $\exists p \spc \| \pmb x (0) \| <_p \delta$.
	Now similarly as in the proof of the Proposition \ref{LyapunovUniformStableLe} but under the consideration of the effects of the application of the Lemma \ref{RealLeLtTrans}:
	\begin{align*}
&	\kappa_1(\|\pmb x(n))\|) \le_r V(\pmb x(n)) \le_r V(\pmb x_0)  \le_r \kappa_2(\| \pmb x(0)\|) \rightarrow \kappa_2(\| \pmb x(0)\|) <_{p_0} \kappa_2(\kappa_2^{-1}(\kappa_1(\varepsilon)))) \rightarrow \\
&	\kappa_1(\|\pmb x(n)\|) <_{p_0+3} \kappa_2(\kappa_2^{-1}(\kappa_1(\varepsilon))) \rightarrow \kappa_2(\kappa_2^{-1} (\kappa_1(\varepsilon))) \le_r \kappa_1(\varepsilon)  \rightarrow \kappa_1(\pmb x_n) <_{p_0+6} \kappa_1(\varepsilon).
	\end{align*}
	And thus, $\kappa_1( \|\pmb x(n) \| ) <_{p_0+6} \kappa_1(\varepsilon)  \rightarrow \exists p \spc \|\pmb x(n)\| <_p \varepsilon$, since $\kappa_2$ is a function of class-$\mathcal{K}_<$.
\end{proof}

The next section focuses on examples where we certify formally the stability of discrete-time dynamical systems.

\section{Example for verification of Lyapunov functions}
\label{sec:benchs}
{\itshape{Example:}} Consider the following  two-dimensional discrete-time rational system:
\begin{align}
\label{ex:3}
x(n+1) & = \frac{y(n)}{1+x(n)^2} \,, \nonumber \\
y(n+1) & = \frac{x(n)}{1+y(n)^2} \,.
\end{align}

One can compute a quadratic SOS Lyapunov $V_2(x,y) = x^2+y^2$ with \texttt{exactLyapunov}, then we may proceed in \texttt{Minlog} as follows.\\
First, we prove the properties that are common for all versions of the Lyapunov theorems:
\begin{enumerate}
	\item \texttt{ProperRHS}: Consider the following functions:
	$f_1,f_2: \mathbf{R}^2 \rightarrow \mathbf{R}$ with the following data:
	\begin{itemize}
		\item $h_{f_1}(\pmb e,n)=e_2,\alpha_{f_{1}}(\pmb e_0,R_0,p)=0,\omega_{f_1}(\pmb e_0,R_0,p)=p+1$\\
		\item $h_{f_2}(\pmb e,n)=1/(1+e_1 \cdot e_1),\alpha_{f_{2}}(\pmb e_0,R_0,p)=0,\omega_{f_2}(\pmb e_0,R_0,p)=p+1$\\
	\end{itemize}
	and construct the map $F:= \left( F_1(x,y)  ,F_2(x,y) \right)$ with $F_1=f_1(x,y)\cdot f_2(x,y)$ and $F_2(x,y)=F_1(y,x)$, where $f_1(x,y)$ and $f_2(x,y)$ are real numbers arising from the program constant . Since application and multiplication produce real numbers and are compatible with the equality, the map $F$ clearly satisfies the requirement \texttt{ProperRHS}.
	\item \texttt{ClassKappa},\texttt{ClassKappaPos}:
	We choose the quadratic function for $\kappa_1,\kappa_2=(\cdot)^2$ and define it in terms of the type \texttt{contex}:
	\begin{itemize}
		\item $h_{(\cdot)^2}(a):=a^2$
		\item $\alpha_{(\cdot)^2}(c,d,p):=0$
		\item $\omega_{(\cdot)^2}(c,d,p):=p+q+1$
	\end{itemize}
	where $q=\texttt{cRatLeAbsBoundPos}(2\cdot\max\left\{|c|,|d|\right\})$ and \texttt{cRatLeAbsBoundPos} is the computational content of the statement $\forall a \in \mathbf{Q} \spc \exists p \in \mathbf{P} \spc  |a|\le 2^p. $ Then we are able to show the required properties for the map $x \mapsto \texttt{Application} \spc (\cdot)^2 x$ as follows.
	\begin{itemize}
		\item \texttt{Contex}: The proof that the tuple $\left(h_{\cdot^2}(a),\alpha_{\cdot^2}(c,d,p)\right)$ is \texttt{Cauchy} for any proper interval $[c,d]$ can be obtained immediately (in an automated way), since the function is polynomial. Next we need to show that for
		\[
		(|a^2-b^2| \le 2^{-p})
		\]
		whenever $a,b \in [c,d], 0\le n$ and $(a-b)\le 2^{-(p+q)}.$ It is easy to see that:
		\begin{align}
			|a^2-b^2| \le |a-b||a+b| \le 2^{-(p+q)}2^{q'} \le 2^{-(p+q)}2^q =2^{-p}
		\end{align}
	with $q'=\texttt{cRatLeAbsBoundPos}(a+b)$ and $2^{q'}\le 2^q$.
		\item $0\le_r x^2$ holds, since $x^2$ is a real number and $0\le_s x^2$, $x=_r 0 \rightarrow x^2=_r 0$ follows immediately by substitution and $0\le_r x \rightarrow 0\le_r y\rightarrow  x\le_r y \rightarrow x^2 \le_r y^2$ since $0 \le_r x,y$.
		\item $0\le_r x \rightarrow 0\le_r y\rightarrow  x^2 \le_r y^2 \rightarrow x\le_r y $, since $\forall z \spc 0\le_r z \rightarrow  z=_r  \sqrt{z^2}$ and $\sqrt{\cdot}$ is monotone. This proof requires the representation of the square root function via completion monad.
		\item $\forall x \in \mathbf{R} \left(  \exists p \spc 0 <_p x \rightarrow \exists p \spc 0 <_p x^2 \right)$. This proof was done in two steps. First, assume that for a continuous map $f_{ex}$ of type \texttt{contex}, we have a witness $\eta$ of type $\mathbf{P}\rightarrow \mathbf{P}$, such that
		\[
		\texttt{H\_eta: }	\forall a \in \mathbf{Q} \spc 0<_{p+1} a \rightarrow 0<_{\eta(p)}f_{ex}(a),
		\]
		i.e. the map will produce positive reals, if the input is a positive rational. The strengthening $p+1$ is required then to prove \[
		\texttt{H\_eta} \rightarrow \forall x \in \mathbf{R} \left(  \exists p \spc 0 <_p x \rightarrow \exists p \spc 0 <_p f_{ex}(x) \right)
		\]
		and takes into account the effects of the Lemma \ref{RealPosChar}. Thus, since for the quadratic function the witness is $\eta(p)=2(p+1)$, \texttt{H\_eta} holds and proves the required property.
		\item $\forall x,y \left( 0\le_r x \rightarrow  \exists p \spc x <_p y \rightarrow \exists p \spc x^2  <_p y^2 \right)$. Given $x=\left((a_n)_n,M\right)$,$x=\left((b_n)_n,N\right)$, $0 \le_r x$ and let $p_0$ be the quantity such that $x<_{p_0} y$. We assert that \[
		\exists n_0,p \spc \forall n\ge n_0 \left(2^{-p_0} \le a_n + b_n\right).
		\]
		The proof of this assertion involves combination of the Lemma \ref{RealPosChar} and \ref{RealNNegChar}. Let $m$ and $p_1$ be the quantities that prove the above assertion and claim that $x^2 <_{(p_0+p_1+2)} y^2$, since 
		\begin{align*}
		\forall n \ge n_1 \spc 2^{-(p_0+p_1+1)} \le b_n^2 - a_n^2 \rightarrow x^2 <_{p_0+p_1+2} y^2
		\end{align*}
	with $n_1=\max\left\{M(p+2),N(p+2),m)\right\}$ by Lemma \ref{RealPosChar}. Since $b_n^2-a_n^2=(b_n-a_n)(b_n+a_n)$,$2^{-(p_0+1)} \le (b_n-a_n)$ by Lemma \ref{RealPosChar} and $2^{-p_1}\le (b_n+a_n)$ by the above assertion, the goal follows.
	\item $\left(0\le_r x \rightarrow 0\le_r y \rightarrow \exists p \spc x^2  <_p y^2 \rightarrow \exists p \spc x <_p y \right)$. Here as a partial goal we prove
	\begin{align*}
		\texttt{RealLtWeaken: }\forall x,y \in \mathbf{R} \spc 0\le_r y \rightarrow \exists p \spc 0<_p xy  \rightarrow \exists p \spc 0<_p x
	\end{align*}
then the required property follows from \texttt{RealLtWeaken}$[(x-y),(x+y)]$.
	\end{itemize}  
	\item \texttt{PDVDiff}: by assertion that $\forall \pmb x \spc V(\pmb F(\pmb x)) \le_s V(\pmb x)$, the desired property follows.
\end{enumerate}
\textbf{Verification in sense of Propositions \ref{LyapunovStableLe},\ref{LyapunovStableLt},\ref{LyapunovStableLtClass}:}
\begin{enumerate}
	\item For the Euclidean norm $\| \cdot \|_2$ we have $\forall \pmb x \in \mathbf{R}^2 \spc  \| \pmb x \|_2^2 \le_r V(\pmb x)$ by the properties of the square root function.
	\item The condition that $\forall \beta>0 \spc \exists \delta \|\pmb x\|_2 \le_r \delta \rightarrow V(\pmb x) \le_r \beta$ can be easily verified by setting $\delta:=\sqrt{\beta}$.
\end{enumerate}
Thus $V,\kappa_1$ and $\pmb F$ satisfy conditions of the propositions \ref{LyapunovStableLe},\ref{LyapunovStableLt},\ref{LyapunovStableLtClass} and thus the equilibrium of \ref{sys:dt} is stable in sense of corresponding definitions.\\
\textbf{Verification in sense of Propositions \ref{LyapunovUniformStableLe},\ref{LyapunovUniformStableLt}:}
\begin{enumerate}
	 \item \texttt{ClassKappaInverse}:
	 Clearly, the square root function is a correct inverse function and is given by the formulae \ref{fct:sqrt} and \ref{fct:sqrt2} via the completion monad \ref{UCFapp}. 
	 \item Apply the Theorem \texttt{LyapunovUniformStable} and parse $V_2,\kappa_1,\kappa_2$ and $\kappa_2^{-1}$.
	 \item \texttt{PDV}:
	 We may choose the Euclidean norm $\| \cdot \|_2$. Then $\kappa_1(\| \pmb x \|_2) \le V(\pmb x) \le \kappa_2(\| \pmb x \|_2)$.
\end{enumerate}
Thus $V,\kappa_1$,$\kappa_2$ and $\pmb F$ satisfy the conditions of the propositions \ref{LyapunovUniformStableLe},\ref{LyapunovUniformStableLt} and thus the equilibrium of \ref{sys:dt} is uniformly stable in sense of corresponding definitions.\\

\section{Conclusion}
We provide  a computer-assisted framework to formally prove the stability of a given discrete-time polynomial system. 
For this, our framework consists of proving formally that a given function is a Lyapunov function, by means of exact SOS certificates.
As a proof of concept, we illustrate our formalization framework on a simple example with two variables. 
In order to apply such techniques to higher dimensional systems, one needs to overcome the typical scalability issues related to SOS-based optimization on polynomial systems. 
One possible remedy is to exploit the sparsity of the input data, either when there are few correlations between the variables \cite{tacchi2019approximating} or when there are few involved  monomials terms \cite{tssos,chordal-tssos,cstssos}.
An alternative to SOS is to rely on other nonnegativity certificates, such as sums of nonnegative circuits or sums of arithmetic-geometric-exponentials, which can be computed with more efficient classes of conic programming, such as second-order cone programming \cite{wang2019second}. 
Exact certificates can also be obtained in this case \cite{magron2019exact}.
An interesting track of further research is to derive formally certified outer/inner approximations of sets of interest arising in the context of dynamical systems, such as backward/forward reachable sets or maximal invariants. 
As for constructive formalization we are interested in expanding our library to the full formalization of multivariate analysis, vector analysis and verified numerical schemes for initial value problems. These concepts build a foundation of important notions in control theory and system dynamics and allow further to investigate the proof-as-programs paradigm using constructive real numbers.  Additionally there exists a verified implementation of DPLL algorithm in \texttt{Minlog} \cite{LAWRENCE2012243} which is of separate interest since many important applications (e.g. SMT solvers, graph algorithms) rely on DPLL and hence program extraction techniques may be applied on such kind of problems.

\paragraph{\textbf{Acknowledgements}.} 
The second author was supported by the FMJH Program PGMO (EPICS project) and  EDF, Thales, Orange et Criteo, as well as from the Tremplin ERC Stg Grant ANR-18-ERC2-0004-01 (T-COPS project).
This work has benefited from the Tremplin ERC Stg Grant ANR-18-ERC2-0004-01 (T-COPS project), the European Union's Horizon 2020 research and innovation programme under the Marie Sklodowska-Curie Actions, grant agreement 813211 (POEMA) as well as from the AI Interdisciplinary Institute ANITI funding, through the French ``Investing for the Future PIA3'' program under the Grant agreement n$^{\circ}$ANR-19-PI3A-0004.

\bibliographystyle{plain}
\bibliography{bib/sos,bib/qe,bib/constr-math,bib/formal_math,bib/form-ver-ctrl,bib/minlog,bib/computable,bib/russian-school,bib/schwichtenberg,bib/computable-functionals,bib/control-theory,bib/cas,bib/dyn-sys,bib/devadze}
\end{document}